%% file: main.tex
\documentclass[12pt]{amsart}
\usepackage[utf8]{inputenc}

\usepackage{geometry}
\usepackage{setspace}
\usepackage{graphicx}
\usepackage{amssymb}
\usepackage{mathrsfs}
\usepackage{amsthm}
\usepackage{amsmath}
\usepackage{color}
\usepackage{caption}
\usepackage{comment}
\usepackage{subcaption}
\usepackage{mathtools}
\usepackage{enumerate}
\usepackage{algorithm,algorithmic}

\usepackage{hyperref}

\theoremstyle{plain} 
\newtheorem{theorem}{Theorem}[section]
\newtheorem{definition}[theorem]{Definition}

\newtheorem{lemma}[theorem]{Lemma}

\theoremstyle{definition} 
\newtheorem{example}[theorem]{Example}
\newtheorem{assumption}[theorem]{Assumption}

\theoremstyle{remark} 
\newtheorem{remark}[theorem]{Remark}

\numberwithin{equation}{section} 

\makeatletter
\newcommand{\myitem}[1]{%
\item[#1]\protected@edef\@currentlabel{#1}%
}
\makeatother

\title[ ]{An $L^0$-approach to stochastic evolution equations}
\author[Ø.~S.~Auestad]{Øyvind S. Auestad}
\address{\newline Department of Mathematical Sciences Norwegian University of Science and Technology, \newline 7034 Trondheim, Norway.} \email{oyvinau@ntnu.no}

\begin{document}

\begin{abstract}
    We introduce a framework for studying pathwise time regularity and numerical approximation of $L^0$-valued stochastic evolution equations. At the core of our framework are two Burkholder--Davis--Gundy type inequalities accommodating Itô integrals with respect to only stochastically integrable processes. The first of these inequalities is formulated in suitable metrics which metrize convergence in probability on the space of integrands and integrals. The second is a modified version, tailored for deriving pathwise properties of the integral. By combining it with a refined version of the Kolmogorov continuity test, we obtain a powerful method for deriving Hölder regularity of Itô integrals in their most general form. Moreover, it provides a simple and powerful way of deriving rates of pathwise convergence of numerical approximations of stochastic evolution equations. Both applications are illustrated for a class of linear parabolic stochastic evolution equations with generalized Whittle--Matérn type noise, and our findings are verified by numerical experiments from this setting. 
\end{abstract}

\date{\today}
\thanks{The research of the author was supported by Grant No. 325114 of the Norwegian Research Council.}

\maketitle

\input{introduction}
\input{stochastic-integration-revisited.tex}
\input{example.tex}

\addcontentsline{toc}{section}{References}
\bibliographystyle{abbrv}
\bibliography{refs.bib}

\appendix
\input{appendix-A.tex}
\input{appendix-B.tex}

\end{document}

%% file: introduction.tex
\section{Introduction}

We study $L^0$-valued Itô integrals and stochastic evolution equations. The main result of this paper are the inequalities of Theorem \ref{theorem:extension-inequality} and \ref{theorem:extension-inequality-2}, which are Burkholder--Davis--Gundy type inequalities for Itô integrals with respect to only stochastically integrable processes. The first of these inequalities is formulated in appropriate metrics which metrize convergence in probability on the space of integrands and integrals, and provides a simple and direct alternative way of extending the definition of the Itô integral from elementary integrands to any stochastically integrable process. In literature, this extension is done by the localization procedure (see, e.g., \cite{da-prato}). The second inequality is a modified version of the first, tailored for deriving pathwise properties of the Itô integral. By combining this inequality with Lemma \ref{lemma:continuity-test} and \ref{lemma:pathwise}, we get simple and powerful methods for deriving pathwise Hölder regularity and rates of convergence of numerical approximations of only $L^0$-valued stochastic evolution equations. In literature, time regularity and convergence rates of numerical approximations of stochastic evolution equations are usually derived using the Burkholder--Davis--Gundy inequality, which requires solutions in $L^2$. 

As an illustrating example, we consider the following parabolic stochastic evolution equation with generalized Whittle--Matérn type noise,
\begin{align}\label{eq:example-model}
    du(t) = -A_1 u(t) \, dt + b(t) A_2^{-\gamma} \, dW(t), \quad u(0) = \xi, \qquad t \in [0,T].
\end{align}
Here, $W$ is a cylindrical Wiener process on $L^2(\mathcal{D})$, $\mathcal{D} \subseteq \mathbb{R}^d$, $d \leq 3$, a bounded polygonal domain, $\xi$ a random variable, $\gamma > d/4 - 1/2$ a parameter, $b$ a scalar stochastic process, and $A_j$, $j = 1, 2$ second order elliptic operators (see Assumption \ref{assumption:model} for details). The model \eqref{eq:example-model} is a generalization of a popular model in applied statistics, studied in \cite{2025-auestad}, in the sense that the noise, $b(t) A_2^{-\gamma} \, dW(t)$, is scaled by the scalar process $b$. While the parameter $\gamma$ controls the spatial regularity of the noise, and hence the time and space regularity of the solution, the process $b$ provides a way to control the statistical regularity, i.e. to which $L^p$ space the solution belongs. Only assuming that $b$ is a predictable process with bounded trajectories, one cannot expect $u(t)$ to be in $L^2$, and so the classical $L^2$-dependent methods for studying pathwise properties of $u$ and its numerical approximation no longer apply. This warrants the development of more powerful methods also accommodating $L^0$-valued stochastic evolution equations, and Theorem \ref{theorem:hölder-regularity} and \ref{theorem:pathwise-rate} describe pathwise Hölder regularity and convergence rates of a fully discrete finite element approximation of \eqref{eq:example-model}.

The paper is structured as follows: in Section 2 we cover preliminaries on stochastic integration. In Section 3 we develop the inequalities of Theorem \ref{theorem:extension-inequality} and \ref{theorem:extension-inequality-2}, and outline an alternative and simple way of extending the definition of the Itô integral to any stochastically integrable process, which is described in Theorem \ref{theorem:extension-inequality}. We then state Lemma \ref{lemma:continuity-test} and \ref{lemma:pathwise}, which are key for deriving pathwise properties of the Itô integral. In Section 4 we apply the framework of Section 3 to study time regularity and numerical approximation of the model problem \eqref{eq:example-model}. Precise conditions on \eqref{eq:example-model} and its numerical approximation are formulated in Assumption \ref{assumption:model} and \ref{assumption:fem}, and Theorem \ref{theorem:hölder-regularity} and \ref{theorem:pathwise-rate} describe time regularity and convergence rates under these conditions. Finally, we verify the convergence rates obtained in Theorem \ref{theorem:pathwise-rate} by numerical experiments. 

%% file: stochastic-integration-revisited.tex
\section{Preliminaries and notation}

Throughout the paper we fix $T > 0$ and a filtered probability space $(\Omega, \mathcal{F},\{\mathcal{F}_t\}_{t \in [0,T]}, P)$, and let $\mathcal{P}_T$ be the corresponding predictable $\sigma$-algebra (the smallest $\sigma$-algebra containing all sets of the form $(s,t] \times F$ for $F \in \mathcal{F}_s$ and $0 \leq s \leq t \leq T$). We abbreviate $\Omega_T := [0,T] \times \Omega$, and denote the product measure on $\Omega_T$ by $P_T$. Let $H$ be a separable Hilbert space, and $W$ be a cylindrical Wiener process on $H$ (covariance operator $I$) adapted to the filtration $\{\mathcal{F}_t\}_{t \in [0,T]}$. Whenever we consider Itô integrals in the following, it will involve this cylindrical Wiener process. 

For another separable Hilbert space $U$ we denote by $L_2(U,H)$ the Hilbert space of Hilbert--Schmidt operators from $U$ to $H$, with inner product,
\begin{align*}
    (A,B)_{L_2(H)} := \sum_{j = 1}^{\infty}(A e_j, B e_j)_H,
\end{align*}
for any orthonormal basis $\{ e_j \}_{j = 1}^{\infty}$ of $U$, and with convention $L_2(H) := L_2(H,H)$. For Banach spaces $X,Y$, we denote by $L(X,Y)$ the Banach space of bounded linear operators from $X$ to $Y$ with the usual norm, and with convention $L(X) := L(X,X)$.

For some $p \geq 0$, we denote by $L^p(\Omega ; X)$ the linear space of equivalence classes of measurable functions $(\Omega, \mathcal{F}) \to (X, \mathcal{B}(X))$, where $E[\Vert x \Vert_X^p] < \infty$. For $p \geq 1$, $L^p(\Omega ; X)$ is a Banach space with norm,
\begin{align*}
    \Vert x \Vert_{L^p(\Omega ; X)}^p := E[\Vert x \Vert_X^p].
\end{align*}
We denote by $C([0,T] ; X)$ the Banach space of continuous functions $[0,T] \to X$, with norm $\Vert x \Vert_{\infty} := \sup_{t \in [0,T]} \Vert x(t) \Vert_X$. For $\beta \in (0,1)$, we denote by $C^{\beta}([0,T] ; X)$ the Banach space of $\beta$-Hölder continuous functions, with norm
\begin{align*}
    \Vert x \Vert_{C^{\beta}} := \Vert x \Vert_{\infty} + \sup_{s,t \in [0,T]}\frac{\Vert x(t) - x(s) \Vert_X}{\vert s - t \vert^{\beta}}.
\end{align*}

Throughout the paper, we will denote by $C$ a generic constant, which may change from line to line. Parameter dependence of $C$ will be denoted by subscripts, but is omitted unless relevant. We also use the notation $x \wedge y := \min(x, y)$. Sometimes (see, e.g., condition \ref{cond:solution-operator} of Assumption \ref{assumption:fem}), without further explanation, we consider bounded linear operators $A \in L(H)$, with the property that $A B \in L(H)$ for some densely defined and possibly unbounded $B : D(B) \to H$. By this we understand that $A B$ extends from $D(B)$ to $H$ as a bounded linear operator, and this extension is denoted by $A B$. 

Finally, we will need the following standard definition and lemma. 
\begin{definition}
    For a metric space $(X,d)$, we denote by $L^0(\Omega; X)$ the space of equivalence classes of measurable functions $(\Omega, \mathcal{F}) \to (X, \mathcal{B}(X))$, whose elements are equal if they differ on sets of $P$-measure $0$.
\end{definition}
It is well known that convergence in probability is metrizable by a metric on $L^0(\Omega ; X)$ (and that this metric is complete if $X$ is complete). Recall that a sequence $\{ x_n \}_n$ of $X$-valued random variables converges in probability to a random variable $x$ if
\begin{align*}
    P(d(x_n,x) > \epsilon) \to 0, \quad \text{for any } \epsilon > 0.
\end{align*}
\begin{lemma}\label{lemma:metric}
    For a complete metric space $(X,d)$ and any $p \geq 1$, convergence in probability is metrizable by the complete metric,
    \begin{align}\label{eq:L0-metric}
        d_p(x,y) := E[1 \wedge d(x,y)^p]^{1/p},
    \end{align}
    on $L^0(\Omega ; X)$. 
\end{lemma}
\begin{proof}
    See Appendix \ref{app:A}.
\end{proof}

\subsection{Stochastic integration}

For a Banach space $X$ and $p \geq 0$, we define
\begin{align*}
    \mathcal{N}_W^p(0,T; X) := \{ \Phi \in L^p(\Omega ; L^2(0,T; X)), \ \Phi : (\Omega_T, \mathcal{P}_T) \to (X, \mathcal{B}(X)) \ \text{is measurable} \},
\end{align*}
and use the notation,
\begin{align*}
    \mathcal{N}_W^2(0,T) := \mathcal{N}_W^2(0,T; L_2(H)), \qquad \mathcal{N}_W(0,T) := \mathcal{N}_W^0(0,T; L_2(H)).
\end{align*}
An elementary integrand $\Phi \in \mathcal{N}_W^2(0,T)$, is of the form, 
\begin{align*}
    \Phi = \sum_{n = 0}^{N-1} \Phi_n \chi_{(t_n, t_{n+1}]}, \qquad \Phi_n \in L^2(\Omega, \mathcal{F}_{t_n}, P; L_2(H)),
\end{align*}
where $N > 0$ is an integer, $\chi_{D}(\cdot)$ is the indicator function on the set $D$, and the Itô integral with respect to $\Phi$ is defined as
\begin{align*}
    \int_0^t \Phi \, dW := \sum_{n = 0}^{N-1} \Phi_n (W(t_{n+1} \wedge t) - W(t_n \wedge t)).
\end{align*}
It may be seen that for $\Phi$ elementary: \smallskip
\begin{enumerate}
    \item[(a)] Itô's isometry holds
    \begin{align}\label{eq:isometry}
        E[\Vert \int_0^t \Phi \, dW \Vert_H^2] = \int_0^t E[\Vert \Phi \Vert_{L_2(H)}^2] \, ds,
    \end{align}
    \item[(b)] for any $p > 0$, there is $C_p > 0$ such that the following Burkholder--Davis--Gundy inequality holds,
    \begin{align}\label{eq:bdg}
        E[ \sup_{t\in [0,T]} \Vert \int_0^t \Phi \, dW \Vert_H^p ] \leq C_p \bigg( \int_0^T E[\Vert \Phi \Vert_{L_2(H)}^2] \, ds \bigg)^{p / 2},
    \end{align}
    (see Theorem 4.36 in \cite{da-prato}), and \medskip
    \item[(c)] $\int_0^{\cdot} \Phi \, dW$ has continuous trajectories $P$-a.s.
\end{enumerate}
An immediate consequence of (a)--(c) above is that the Itô integral extends (by continuity) to $\mathcal{N}_W^2(0,T)$, and (a)--(c) also holds for integrals with respect to integrands in $\mathcal{N}_W^2(0,T)$. The procedure of extending the integral to $\mathcal{N}_W(0,T)$, which in literature is referred to as ``stochastically integrable processes", is called the localization procedure (see, e.g., Chapter 4 in \cite{da-prato}). In Theorem \ref{theorem:extension-inequality} we provide an alternative way of performing this extension, in the setting of Lemma \ref{lemma:metric}. 

\section{Framework}

The following two theorems are the main results of this paper, and contain Burkholder--Davis--Gundy type inequalities for Itô integrals with values in $L^0$. 
\begin{theorem}\label{theorem:extension-inequality}
    The Itô integral,
    \begin{align*}
        \Phi \mapsto \int_0^{\cdot} \Phi \, dW,
    \end{align*}
    with $\Phi$ elementary, extends by continuity to a continuous linear mapping,
    \begin{align*}
        \mathcal{N}_W(0,T) \to L^0(\Omega ; C([0,T] ; H)) \cap \mathcal{N}_W^0(0,T ; H),
    \end{align*}
    where the topologies on $\mathcal{N}_W(0,T)$ and $L^0(\Omega ; C([0,T] ; H)) \cap \mathcal{N}_W^0(0,T; H)$ are those given by the metric \eqref{eq:L0-metric}, i.e. convergence in probability of random variables in $L^2(0,T ; L_2(H))$ and $C([0,T] ; H)$, respectively. For any $\Phi \in \mathcal{N}_W(0,T)$ and $p > 0$, there is $C_p > 0$ such that
    \begin{align}\label{eq:L0-bdg}
        E[1 \wedge \sup_{t \in [0,T]} \Vert \int_0^t \Phi \, dW \Vert_H^p] \leq C_p E[1 \wedge \int_0^T \Vert \Phi \Vert_{L_2(H)}^2 \, dt]^{p/2}.
    \end{align}
\end{theorem}
\begin{proof}
    It suffices to show that the inequality of the theorem holds for $\Phi$ elementary, since elementary integrands are dense in $\mathcal{N}_W(0,T)$, $\int_0^{\cdot} \Phi \, dW \in L^0(\Omega ; C([0,T] ; H)) \cap \mathcal{N}_W^0(0,T; H)$ for $\Phi$ elementary, and the metrics,
    \begin{align*}
        d(\Phi_1, \Phi_2) = E[1 \wedge \int_0^T \Vert \Phi_1 - \Phi_2 \Vert_{L_2(H)}^2 \, ds]^{1/2}, \quad \Phi_1,\Phi_2 \in \mathcal{N}_W(0,T),
    \end{align*}
    and
    \begin{align*}
        d(x, y) = E[1 \wedge \sup_{t \in [0,T]} \Vert x(t) - y(t) \Vert_H^2]^{1/2}, \quad x,y \in L^0(\Omega ; C([0,T] ; H)) \cap \mathcal{N}_W^0(0,T ; H),
    \end{align*}
    define complete metrics on $\mathcal{N}_W(0,T)$ and $L^0(\Omega ; C([0,T] ; H)) \cap \mathcal{N}_W^0(0,T ; H)$, respectively, due to Lemma \ref{lemma:metric}. Note in particular that convergence in the metrics above imply convergence in $P_T$-measure of random variables with values in $L_2(H)$ and $H$, respectively, and so any limit will be $\mathcal{P}_T$-measurable. Thus, the extension by continuity of the integral follows by the inequality of the theorem with $p = 2$.

    We now show the inequality for $\Phi$ elementary, starting with the case $p \in (0,2]$. To that end, define
    \begin{align*}
        Z(t) := \int_0^t \Vert \Phi \Vert_{L_2(H)}^2 \, ds,
    \end{align*}
    and the stopping time
    \begin{align*}
        \tau := \inf \{ t \in [0,T] \ : \ Z(t) \geq 1 \},
    \end{align*}
    with convention $\inf \emptyset := T$. We have
    \begin{align*}
        E[1 \wedge \sup_{t \in [0,T]} \Vert \int_0^t \Phi \, dW \Vert_H^p] &= E[ 1 \wedge  \sup_{t \in [0,T]} \Vert \int_0^t \Phi \, dW \Vert_H^p (\chi_{\{ Z(T) < 1\}} + \chi_{\{ Z(T) \geq 1\}} )] \\
        &\leq E[1 \wedge \sup_{t \in [0,T]} \Vert \int_0^t \Phi \, dW \Vert_H^p \chi_{\{ Z(T) < 1\}} ] + E[\chi_{\{ Z(T) \geq 1\}}] \\
        &=: (i) + (ii).
    \end{align*}
    For $(i)$, we get using that $\tau = T$ when restricted to the set $\{ Z(T) < 1 \}$, $\chi_{[0,\tau]}\Phi \in \mathcal{N}_W^2(0,T)$ and the Burkholder--Davis--Gundy inequality \eqref{eq:bdg}
    \begin{align*}
        E[1 \wedge \sup_{t \in [0,T]} \Vert \int_0^t \Phi \, dW \Vert_H^p \chi_{\{ Z(T) < 1\}} ] &\leq E[1 \wedge \sup_{t \in [0,T]} \Vert \int_0^t \chi_{[0,\tau]} \Phi \, dW \Vert_H^p ] \\
        &\leq E[\sup_{t \in [0,T]} \Vert \int_0^t \chi_{[0,\tau]} \Phi \, dW \Vert_H^p ] \\
        &\leq C_p E[ \int_0^T \Vert \chi_{[0,\tau]} \Phi \Vert_{L_2(H)}^2 \, dt ]^{p / 2} \\
        &= C_p E[1 \wedge \int_0^T \Vert \Phi \Vert_{L_2(H)}^2 \, dt ]^{p / 2}.
    \end{align*}
    For $(ii)$, we note that
    \begin{align*}
        E[\chi_{\{ Z(T) > 1 \}}] \leq E[1 \wedge Z(T)] = E[1 \wedge \int_0^T \Vert \Phi \Vert_{L_2(H)}^2 \, dt] \leq E[1 \wedge \int_0^T \Vert \Phi \Vert_{L_2(H)}^2 \, dt]^{p / 2}.
    \end{align*}
    Combining these terms in the case of $p \in (0,2]$, we get the estimate for these powers. 
    
    For $p > 2$, define
    \begin{align*}
        X(t) := \int_0^t \Phi \, dW,
    \end{align*}
    the stopping time (with convention $\inf \emptyset := T$) and stopped process
    \begin{align*}
        \tau' := \inf \{ t \in [0,T] \ : \ \Vert X(t) \Vert_H \geq 1 \}, \quad X'(t) := \int_0^t \chi_{[0,\tau']} \Phi \, dW,
    \end{align*}
    and note in particular that $1 \wedge \Vert X(t) \Vert_H \leq \Vert X'(t) \Vert_H$ and $1 \wedge \Vert X'(t) \Vert_H = \Vert X'(t) \Vert_H$. By the Burkholder--Davis--Gundy inequality \eqref{eq:bdg}, we have
    \begin{align*}
        E[1 \wedge \sup_{t \in [0,T]} \Vert X(t) \Vert_H^p] \leq E[\sup_{t \in [0,T]} \Vert X'(t) \Vert_H^p] \leq C_p E[\int_0^T \Vert \chi_{[0,\tau']} \Phi \Vert_{L_2(H)}^2 \, ds]^{p / 2}.
    \end{align*}
    Owing to Itô's isometry and the case $p = 2$ shown above, we get
    \begin{align*}
        E[\int_0^T \Vert \Phi \chi_{[0,\tau']} \Vert_{L_2(H)}^2 \, ds] &= E[\Vert X'(T) \Vert_H^2] \\
        &= E[1 \wedge \Vert X'(T) \Vert_H^2] \\ 
        &\leq C E[1 \wedge \int_0^T \Vert \Phi \chi_{[0,\tau']} \Vert_{L_2(H)}^2 \, ds ] \\
        &\leq C E[1 \wedge \int_0^T \Vert \Phi \Vert_{L_2(H)}^2 \, ds ],
    \end{align*}
    which when inserted into the expression above gives the inequality for $p > 2$.
\end{proof}

The following theorem is a modified version of the inequality of Theorem \ref{theorem:extension-inequality} above, and will be key for deriving pathwise properties of the Itô integral.
\begin{theorem}\label{theorem:extension-inequality-2}
    For any sequence $\{\Phi_n \}_n \subseteq \mathcal{N}_W(0,T)$ and $p \geq 2$, there is $C_p > 0$, such that
    \begin{align*}
        E[1 \wedge \sum_n \sup_{t\in[0,T]} \Vert \int_0^t \Phi_n \, dW \Vert_H^p] \leq C_p E[1 \wedge \sum_n \bigg( \int_0^T \Vert \Phi_n \Vert_{L_2(H)}^2 \, dt\bigg)^{p/2}].
    \end{align*}
\end{theorem}
\begin{proof}
    The proof is similar to that of Theorem \ref{theorem:extension-inequality}. Assume for the moment that $\Phi_n$ are elementary, as general $\Phi_n$ follows by density. Define the process and stopping time
    \begin{align*}
        Z(t) := \sum_n \bigg( \int_0^t \Vert \Phi_n \Vert_{L_2(H)}^2 \, ds \bigg)^{p / 2}, \quad \tau := \inf \{ t \in [0,T] \ : \ Z(t) \geq 1 \},
    \end{align*}
    with convention $\inf \emptyset := T$. Arguing as in the proof of Theorem \ref{theorem:extension-inequality}, we have
    \begin{align*}
        E[1 \wedge \sum_n \sup_{t \in [0,T]} \Vert \int_0^t \Phi_n \, dW \Vert_H^p] &= E[1 \wedge \sum_n \sup_{t \in [0,T]} \Vert \int_0^t \Phi_n \, dW \Vert_H^p ( \chi_{\{Z(T) < 1\}} + \chi_{\{Z(T) \geq 1\} } )] \\
        &\leq E[1 \wedge \sum_n \sup_{t \in [0,T]} \Vert \int_0^t \Phi_n \, dW \Vert_H^p \chi_{\{ Z(T) < 1\}}] + E[\chi_{\{Z(T) \geq 1\}}] \\
        &=: (i) + (ii).
    \end{align*}
    For $(i)$ we have, using that $\tau = T$ when restricted to the set $\{ Z(T) < 1 \}$, $\chi_{[0,\tau]}\Phi_n \in \mathcal{N}_W^2(0,T)$, the Burkholder--Davis--Gundy inequality \eqref{eq:bdg}, and Hölder's inequality
    \begin{align*}
        E[1 \wedge \sum_n \sup_{t \in [0,T]} \Vert \int_0^t \Phi_n \, dW \Vert_H^p \chi_{\{ Z(T) < 1\}}] &\leq E[1 \wedge \sum_n \sup_{t \in [0,T]} \Vert \int_0^t \chi_{[0,\tau]} \Phi_n \, dW \Vert_H^p] \\
        &\leq E[\sum_n \sup_{t \in [0,T]} \Vert \int_0^t \chi_{[0,\tau]} \Phi_n \, dW \Vert_H^p] \\
        &= \sum_n E[\sup_{t \in [0,T]} \Vert \int_0^t \chi_{[0,\tau]} \Phi_n \, dW \Vert_H^p] \\
        &\leq C_p \sum_n E[\int_0^T \Vert \chi_{[0,\tau]} \Phi_n \Vert_{L_2(H)}^2 \, dt]^{p / 2} \\
        &\leq C_p \sum_n E[\bigg( \int_0^T \Vert \chi_{[0,\tau]} \Phi_n \Vert_{L_2(H)}^2 \, dt \bigg)^{p / 2}] \\
        &= C_p E[\sum_n \bigg( \int_0^T \Vert \chi_{[0,\tau]} \Phi_n \Vert_{L_2(H)}^2 \, dt \bigg)^{p / 2}] \\
        &= C_p E[1 \wedge \sum_n \bigg( \int_0^T \Vert \Phi_n \Vert_{L_2(H)}^2 \, dt \bigg)^{p / 2}].
    \end{align*}
    For $(ii)$, we get
    \begin{align*}
        E[\chi_{\{ Z(T) \geq 1 \}}] \leq E[1 \wedge Z(T)] = E[1 \wedge \sum_n \bigg( \int_0^T \Vert \Phi_n \Vert_{L_2(H)}^2 \, dt \bigg)^{p / 2}].
    \end{align*}
\end{proof}

The following is a version of the Kolmogorov continuity test which accommodates stochastic processes that are only in $L^0$, and is tailored for the inequality of Theorem \ref{theorem:extension-inequality-2}.
\begin{lemma}\label{lemma:continuity-test}
    Let $(X,d)$ be a complete metric space, and $x : [0,T] \times \Omega \to X$ a stochastic process. If for some $\delta, \beta > 0$, 
    \begin{align*}
        E[1 \wedge \sum_{m \geq n} \sum_{j = 0}^{2^m - 1} \frac{d(x(j T 2^{-m}), x((j+1)T2^{-m}))^{\delta}}{(T2^{-m})^{\beta \delta}}] \to 0 \quad \text{as} \quad n \to \infty,
    \end{align*}
    the trajectories of $x$ are $\beta$-Hölder continuous $P$-a.s.
\end{lemma}
\begin{proof}
    The proof is similar to, e.g., the proof of Theorem 3.3 in \cite{da-prato}, but with a few modifications. See Appendix \ref{app:A} for the details. 
\end{proof}

\begin{remark}
    From Lemma \ref{lemma:continuity-test} we can deduce the following (weaker) continuity test, which is more analogous to the usual one (see, e.g., Theorem 3.3 in \cite{da-prato}): if for some $\delta, \epsilon, C > 0$, and $t, s \in [0,T]$,
    \begin{align*}
        E[1 \wedge d(x(t), x(s))^{\delta}] \leq C \vert t-s \vert^{1 + \epsilon}, 
    \end{align*}
    then for any $\beta < \epsilon / \delta$, $x$ has $\beta$-Hölder continuous trajectories $P$-a.s. However, this test will not be able to detect any Hölder continuity in the example we consider in the next section. 
\end{remark}

The next lemma is an extension of the technique usually used in order to derive rates of pathwise convergence from convergence in $L^p$, $p \geq 1$, and accommodates random variables in $L^0$.
\begin{lemma}\label{lemma:pathwise}
    Let $x_h : \Omega \to [0,\infty)$, $h \in (0,1)$, be a collection of random variables, and suppose that for some $p > 0$ and sequence $\{h_n\}_n$
    \begin{align*}
        E[1 \wedge \sum_{m \geq n} \frac{x_{h_m}^p}{h_m^p}] \to 0 \quad \text{as } n \to \infty.
    \end{align*}
    Then there exists a random variable $M : \Omega \to [0,\infty)$, independent of $h_n$, ensuring that $x_{h_n} \leq M h_n$ $P$-a.s.
\end{lemma}
\begin{proof}
    Note that for any integer $n \geq 0$
    \begin{align*}
        P(\limsup_m \{ \frac{x_{h_m}}{h_m} > 1 \}) &\leq P(\{\sup_{m \geq n} \frac{x_{h_m}^p}{h_m^p} > 1\}) \\
        &\leq P(\{\sum_{m \geq n} \frac{x_{h_m}^p}{h_m^p} > 1\}) \\
        &\leq E[1 \wedge \sum_{m \geq n} \frac{x_{h_m}^p}{h_m^p}] \to 0,
    \end{align*}
    and so $P(\limsup_n \{ \frac{x_{h_n}}{h_n} > 1 \}) = 0$. Therefore $\sup_n \frac{x_{h_n}}{h_n} < \infty$ $P$-a.s., and setting
    \begin{align*}
        M(\omega) :=
        \begin{cases}
            \sup_n \frac{x_{h_n}}{h_n}, \quad &\text{when }\sup_n \frac{x_{h_n}}{h_n} < \infty, \\
            0, \quad &\text{otherwise},
        \end{cases}
    \end{align*}
    we find $x_{h_n} \leq M h_n$ $P$-a.s.
\end{proof}
With Lemma \ref{lemma:continuity-test} and \ref{lemma:pathwise} at hand, we are ready to study Hölder regularity and convergence rates of numerical approximations of stochastic evolution equations with values in $L^0$. The idea is to use Theorem \ref{theorem:extension-inequality-2} combined with Lemma \ref{lemma:continuity-test} and Lemma \ref{lemma:pathwise}, respectively.

%% file: example.tex
\section{Example: a linear parabolic SPDE with generalized Whittle--Matérn noise}\label{section:example}

In this section we apply the framework of the previous section to the linear parabolic stochastic evolution equation \eqref{eq:example-model}, where the noise is scaled by a scalar process, ensuring that it is statistically irregular enough for the $L^2$-dependent approaches to no longer apply.

Let $H := L^2(\mathcal{D})$. The operators $A_j : D(A_j) \to H$ in \eqref{eq:example-model} are related to sesquilinear forms $a_j : V \times V \to \mathbb{C}$, by $(A_j u, v)_H = a_j(u, v)$ for any $u \in D(A_j) \subseteq V, v \in V$, with $V \subseteq H^1(\mathcal{D})$ some closed subspace. Here, $H^1(\mathcal{D})$ is defined in the usual way as the completion of $C^1(\overline{\mathcal{D}})$ using the $H^1$-norm, given by $\Vert u \Vert_{H^1}^2 := \Vert u \Vert_{L^2}^2 + \Vert \nabla u \Vert_{L^2}^2$. The sesquilinear forms are given by 
\begin{align*}
    a_j(u, v) := \int_{\mathcal{D}} \mathcal{A}_j \nabla u \cdot \nabla \overline{v} + (b_j \cdot \nabla u) \overline{v} + \alpha_j u \overline{v} \, d x, \quad j = 1, 2,
\end{align*} 
with coefficients $\mathcal{A}_j(x) \in \mathbb{R}^{d \times d}$, $b_j(x) \in \mathbb{R}^d$, $\alpha_j(x) \in \mathbb{R}$, where $\overline{v}$ is the complex conjugate of $v$. 

We make the following assumption on \eqref{eq:example-model}.
\begin{assumption}\label{assumption:model}
\hfill
\begin{description}
    \myitem{(M1)}\label{cond:continuity-coercivity} There is $\lambda \geq 0$ such that the shifted sesquilinear form $\lambda (\cdot,\cdot)_H + a_1(\cdot, \cdot)$ and the sesquilinear form $a_2(\cdot, \cdot)$ are coercive and continuous on $V$. \medskip
    \myitem{(M2)}\label{cond:b} $b : \Omega_T \to \mathbb{R}$ is $\mathcal{P}_T$-measurable, $b$ has bounded trajectories and $\Vert b \Vert_{\infty} : \Omega \to \mathbb{R}$ is $\mathcal{F}_T$-measurable. \medskip
    \myitem{(M3)}\label{cond:hilbert-schmidt} For any $\alpha < -d/4$, there is $C_{\alpha} > 0$, such that $\Vert A_2^{\alpha} \Vert_{L_2(H)} \leq C_{\alpha}$. \medskip
    \myitem{(M4)}\label{cond:A1A2} For some $C > 0$, $\Vert (\lambda + A_1) A_2^{-1} \Vert_{L(H)} \leq C$ and $\Vert (\lambda + A_1)^{-1/2} A_2^{1/2} \Vert_{L(H)} \leq C$.
    \medskip
    \myitem{(M5)}\label{cond:gamma-size} $\gamma > d/4 - 1/2$. \medskip
    \myitem{(M6)} $\xi : \Omega \to H$ is $\mathcal{F}_0$-measurable.
\end{description}
\end{assumption}
\begin{remark}
    Sufficient conditions for \ref{cond:continuity-coercivity}, \ref{cond:hilbert-schmidt} and \ref{cond:A1A2} to hold are described in detail in Remark 2.3 of \cite{2025-auestad}. The condition \ref{cond:continuity-coercivity} on $a_1$ is the same as Gårding's inequality, and holds under positivity and boundedness conditions on the coefficients of $a_1$. 
\end{remark}
Under Assumption \ref{assumption:model}, the operators $-A_j : D(A_j) \to H$ generate analytic semigroups on $H$, which we denote by $S_j(\cdot)$, $j = 1, 2$. Fractional powers of $\lambda + A_1$ are defined as 
\begin{align}\label{eq:fractional-powers}
\begin{split}
    (\lambda + A_1)^{-\alpha} &:= \frac{1}{\Gamma(\alpha)} \int_0^{\infty} t^{-1 + \alpha} e^{-\lambda t} S_1(t) \, dt, \\
    (\lambda + A_1)^{\alpha} &:= \frac{\sin(\alpha \pi)}{\pi} \int_0^{\infty} t^{-1 + \alpha} (t + \lambda + A_1)^{-1} (\lambda + A_1) \, dt, 
\end{split}
\end{align}
for $\alpha \in (0,1)$, (the expression for $(\lambda + A_1)^{-\alpha}$ also holds for any $\alpha \geq 1$), and similarly for $A_2$ (replacing $\lambda + A_1$ by $A_2$ and $e^{-\lambda \, \cdot} S_1(\cdot)$ by $S_2(\cdot)$ in \eqref{eq:fractional-powers}). Lemma \ref{lemma:analytic-semigroup} in Appendix \ref{app:B} summarizes some useful properties of analytic semigroups and their generators. 

Under Assumption \ref{assumption:model}, \eqref{eq:example-model} has a mild solution, $u$, defined by the stochastic convolution,
\begin{align}\label{eq:mild-form}
    u(t) := S_1(t) \xi + \int_0^t S_1(t-s) b(s) A_2^{-\gamma} \, dW, \quad t \in [0,T].
\end{align}
To show that $u$ is well defined, and to describe the space and time regularity of $u$, the following lemma will be helpful.
\begin{lemma}\label{lemma:helpful-estiamtes}
    Under Assumption \ref{assumption:model}, the following estimates hold:
    \begin{enumerate}[(a)]
        \item for any $\alpha \geq 0$, $\beta \in [-1/2, \infty)$, and $\epsilon > 0$ small, there is $C_{\epsilon} > 0$ such that
        \begin{align*}
            \Vert (\lambda + A_1)^{\alpha} S_1(t) A_2^{-\beta} \Vert_{L(H)} \leq C_{\epsilon} e^{\lambda t} t^{\min(0,-\alpha - \epsilon + \min(1,\beta))}, 
        \end{align*}
        \item while for any $0 \leq t_1 \leq t_2 \leq T$, and any $\alpha \in [0, \min(1/2, 1 / 2 + \gamma - d / 4))$, there is $C_{\alpha} > 0$ such that
        \begin{align*}
            \int_{t_1}^{t_2} \Vert S_1(t_2 - s) A_2^{-\gamma} \Vert_{L_2(H)}^2 \, ds \leq C_{\alpha} e^{2 \lambda t_2} (t_2 - t_1)^{2 \alpha}, 
        \end{align*}
        while for $\alpha \in [0, \min(1, 1/2 + \gamma - d / 4))$, there is $C_{\alpha} > 0$ such that
        \begin{align*}
            \int_0^{t_1} \Vert (S_1(t_2 - s) - S_1(t_1 - s)) A_2^{-\gamma} \Vert_{L_2(H)}^2 \, ds \leq C_{\alpha} e^{2 \lambda t_2} (t_2 - t_1)^{2 \alpha}.
        \end{align*} 
    \end{enumerate}
\end{lemma}
\begin{proof}
    The estimate $(a)$ follows by Lemma \ref{lemma:analytic-semigroup} and \ref{cond:A1A2}, noting first that
    \begin{align*}
        \Vert (\lambda + A_1)^{\alpha} S_1(t) \Vert_{L(H)} &\leq C e^{(\lambda - \delta) t} t^{-\alpha}, \\
        \Vert (\lambda + A_1)^{\alpha} S_1(t) A_2^{1/2} \Vert_{L(H)} &\leq \Vert (\lambda + A_1)^{\alpha + 1/2} S_1(t) \Vert_{L(H)} \Vert (\lambda + A_1)^{-1/2} A_2^{1/2} \Vert_{L(H)} \\ 
        &\leq C e^{(\lambda - \delta) t} t^{-\alpha-1/2},
    \end{align*}
    which by interpolation (see, e.g., Lemma 2.5 in \cite{2024-auestad}) gives the estimate for $\beta \in [-1/2,0]$. For $\beta \in (0,1]$, note first that from the definition \eqref{eq:fractional-powers} one sees that for any $\epsilon \in (0,\beta)$
    \begin{align*}
        \Vert (\lambda + A_1)^{\beta - \epsilon} A_2^{-\beta} \Vert_{L(H)} &\leq C \Vert \int_0^{\infty} t^{-1 + \beta} (\lambda + A_1)^{\beta - \epsilon} S_2(t) \, dt \Vert_{L(H)} \\
        &\leq C \int_0^{\infty} t^{-1 + \beta} \Vert (\lambda + A_1)^{\beta - \epsilon} S_2(t) \Vert_{L(H)} \, dt \\
        &\leq C \int_0^{\infty} t^{-1 + \epsilon} e^{-\delta t} \, dt \\
        &\leq C_{\epsilon},
    \end{align*}
    where we used that fractional powers of $\lambda + A_1$ are closed to pass it under the integral, and the inequality $\Vert (\lambda + A_1)^{\alpha} S_2(t) \Vert_{L(H)} \leq C e^{-\delta t} t^{-\alpha}$, $\alpha \in [0,1]$, which follows by \ref{cond:A1A2} and interpolation. Hence, using Lemma \ref{lemma:analytic-semigroup} and the observations above
    \begin{align*}
        \Vert A_1^{\alpha} S_1(t) A_2^{-\beta} \Vert_{L(H)} = \Vert A_1^{\alpha - \beta + \epsilon} S_1(t) A_1^{\beta - \epsilon} A_2^{-\beta} \Vert_{L(H)} &\leq \Vert A_1^{\alpha - \beta + \epsilon} S_1(t) \Vert_{L(H)} \Vert A_1^{\beta - \epsilon} A_2^{-\beta} \Vert_{L(H)} \\
        &\leq C_{\epsilon} e^{(\lambda - \delta) t} t^{\min(0,-\alpha + \beta - \epsilon)},
    \end{align*}
    where we used that negative fractional powers of $\lambda + A_1$ are bounded. The cases $\beta > 1$ follows by the boundedness of negative fractional powers of $A_2$.
    
    The estimates in $(b)$ follows similarly by Lemma \ref{lemma:analytic-semigroup} and \ref{cond:A1A2}. See, e.g., the proof of Lemma 2.6 in \cite{2025-auestad} for details. 
\end{proof}

The next lemma asserts that \eqref{eq:example-model} has a mild solution, and describes its spatial regularity.   
\begin{lemma}
    Under Assumption \ref{assumption:model}, \eqref{eq:example-model} has a mild solution $u$ with $u(t) \in D((\lambda + A_1)^{\alpha})$ $P$-a.s. for any $t > 0$ and $\alpha < \min(3/2, 1/2 + \gamma - d / 4)$. 
\end{lemma}
\begin{proof}
    It suffices to show that $(\lambda + A_1)^{\alpha} S_1(t - \cdot) b A_2^{-\gamma} \in \mathcal{N}_W(0,t)$ for any $t \in [0,T]$. Thus, we need $(\lambda + A_1)^{\alpha} S_1(t - \cdot) b A_2^{-\gamma} \in L^2(0,t ; L_2(H))$ $P$-a.s., and a direct computation yields
    \begin{align*}
        \Vert (\lambda + A_1)^{\alpha} S_1(t-\cdot) b A_2^{-\gamma} \Vert_{L_2(H)} \leq \Vert b \Vert_{\infty} \Vert (\lambda + A_1)^{\alpha} S_1(t-\cdot) A_2^{-\gamma} \Vert_{L_2(H)},
    \end{align*} 
    (using that scaling by $b(s)$ viewed as an operator $H \to H$ has operator norm bounded by $\Vert b \Vert_{\infty}$). The first factor is finite $P$-a.s. by \ref{cond:b}, while the second is in $L^2(0,t ; \mathbb{R})$ by Lemma \ref{lemma:helpful-estiamtes} and \ref{cond:hilbert-schmidt}: for any $\epsilon \in (0, \gamma - d / 4 + 1/2]$ small
    \begin{align*}
        \Vert (\lambda + A_1)^{\alpha} S_1(t) A_2^{-\gamma} \Vert_{L_2(H)} &\leq \Vert (\lambda + A_1)^{\alpha} S_1(t) A_2^{-\gamma + d / 4 + \epsilon} \Vert_{L(H)} \Vert A_2^{-d / 4 - \epsilon} \Vert_{L_2(H)} \\
        &\leq C_{\epsilon} e^{\lambda t} t^{\min(0,-\alpha -\epsilon + \min(1, \gamma - d / 4 - \epsilon))},
    \end{align*}
    and so the square of $t^{\min(0,-\alpha - \epsilon + \min(1, \gamma - d / 4 - \epsilon))}$ is integrable if we choose $\alpha < \min(3/2, 1 / 2 + \gamma - d / 4 - \epsilon) - \epsilon$, where we may choose $\epsilon$ arbitrarily small. 
\end{proof}

\subsection{Hölder regularity of the trajectories of (\ref{eq:example-model})}

By combining the following lemma with Theorem \ref{theorem:extension-inequality-2} and Lemma \ref{lemma:continuity-test} we can derive Hölder regularity of the trajectories of the mild solution $u$. 
\begin{lemma}\label{lemma:hölder-estimate}
    Suppose Assumption \ref{assumption:model} holds, and let $u$ be the mild solution to \eqref{eq:example-model} with $\xi = 0$. Then, for $0 \leq t_1 \leq t_2 \leq T$, 
    \begin{align*}
        u(t_2) - u(t_1) = \int_0^T \Phi \, dW, \quad \Phi := \chi_{[0,t_2]} S_1(t_2 - \cdot) b A_2^{-\gamma} - \chi_{[0,t_1]} S_1(t_1 - \cdot) b A_2^{-\gamma}
    \end{align*}
    where for any $\alpha \in [0, \min(1/2, 1/2 + \gamma - d / 4))$, there is $C_{\alpha} > 0$, such that
    \begin{align*}
        \int_0^T \Vert \Phi \Vert_{L_2(H)}^2 \, ds \leq C_{\alpha} \Vert b \Vert_{\infty}^2 e^{2\lambda t_2} (t_2 - t_1)^{2\alpha}.
    \end{align*}
\end{lemma}
\begin{proof}
    Note that
    \begin{align*}
        \Phi &= \chi_{[t_1, t_2]} S_1(t_2 - \cdot) b A_2^{-\gamma} + (S_1(t_2 - \cdot) - S_1(t_1 - \cdot)) \chi_{[0,t_1]} b A_2^{-\gamma} \\
        &=: (i) + (ii),
    \end{align*}
    and that by Young's inequality
    \begin{align*}
        \int_0^T \Vert \Phi \Vert_{L_2(H)}^2 \, ds \leq 2 \int_0^T (\Vert (i) \Vert_{L_2(H)}^2 + \Vert (ii) \Vert_{L_2(H)}^2) \, ds.
    \end{align*}
    By Lemma \ref{lemma:helpful-estiamtes} we have
    \begin{align*}
        \int_0^T \Vert (i) \Vert_{L_2(H)}^2 \, ds &= \int_{t_1}^{t_2} \Vert S_1(t_2 - s) b(s) A_2^{-\gamma} \Vert_{L_2(H)}^2 \, ds \\
        &\leq \Vert b \Vert_{\infty}^2 \int_{t_1}^{t_2} \Vert S_1(t_2 - s) A_2^{-\gamma} \Vert_{L_2(H)}^2 \, ds \\
        &\leq C_{\alpha} \Vert b \Vert_{\infty}^2 e^{2\lambda t_2} (t_2 - t_1)^{2\alpha}.
    \end{align*}
    For $(ii)$, also using Lemma \ref{lemma:helpful-estiamtes},
    \begin{align*}
        \int_0^T \Vert (ii) \Vert_{L_2(H)}^2 \, ds &= \int_0^{t_1} \Vert (S_1(t_2 - s) - S_1(t_1 - s)) b(s) A_2^{-\gamma} \Vert_{L_2(H)}^2 \, ds \\
        &\leq \Vert b \Vert_{\infty}^2 \int_0^{t_1} \Vert (S_1(t_2 - s) - S_1(t_1 - s)) A_2^{-\gamma} \Vert_{L_2(H)}^2 \, ds \\
        &\leq C_{\alpha} \Vert b \Vert_{\infty}^2 e^{2 \lambda t_2} (t_2 - t_1)^{2\alpha}.
    \end{align*}
\end{proof}

Now we are ready to derive Hölder regularity of the trajectories of \eqref{eq:example-model}.
\begin{theorem}\label{theorem:hölder-regularity}
    Suppose Assumption \ref{assumption:model} holds. Then for any $0 < t_0 < T$ and $\beta \in [0, \min(1/2, 1/2 + \gamma - d / 4))$, the trajectories of $u$ are in $C^{\beta}([t_0, T] ; H)$ $P$-a.s. If in addition $\xi \in D((\lambda + A_1)^{\beta})$ $P$-a.s. we may choose $t_0 = 0$.
\end{theorem}
\begin{proof}
    Note first that $S(\cdot) \xi$ is Lipchitz $P$-a.s. on the interval $[t_0, T]$, $0 < t_0 < T$, by Lemma \ref{lemma:analytic-semigroup}. If in addition $\xi \in D((\lambda + A_1)^{\beta})$ $P$-a.s., $S(\cdot) \xi$ is $\beta$-Hölder continuous on the interval $[0,T]$ $P$-a.s. Therefore, it suffices to study the Hölder regularity of $u$ with $\xi = 0$, i.e. the stochastic convolution
    \begin{align*}
        u = \int_0^{\cdot} S_1(\cdot - s) b(s) A_2^{-\gamma} \, dW.
    \end{align*}
    By Theorem \ref{theorem:extension-inequality-2} combined with Lemma \ref{lemma:hölder-estimate} we have for any $\alpha \in [0, \min(1/2, 1/2 + \gamma - d / 4))$
    \begin{align*}
        &E[1 \wedge \sum_{m \geq n} \sum_{j = 0}^{2^m - 1} \frac{\Vert u(j T 2^{-m}) - u((j+1) T 2^{-m}) \Vert_H^p}{(T 2^{-m})^{\beta p}}] \\
        &\qquad \leq E[1 \wedge \sum_{m \geq n} \sum_{j = 0}^{2^m - 1} \bigg( \frac{C_{\alpha} e^{2\lambda T} (T 2^{-m})^{2\alpha} \Vert b \Vert_{\infty}^2}{ (T 2^{-m})^{2\beta}} \bigg)^{p / 2}] \\
        &\qquad \leq E[1 \wedge C_{p,\alpha} e^{p \lambda T} \Vert b \Vert_{\infty}^p \sum_{m \geq n} 2^m (T 2^{-m})^{p(\alpha - \beta)} ] \to 0,
    \end{align*}
    as $n \to \infty$ by dominated convergence, provided $\beta < \alpha$, as we may choose $p \geq 2$ as large as desired. Hence, by Lemma \ref{lemma:continuity-test}, the stochastic convolution has trajectories in $C^{\beta}([0,T] ; H)$ $P$-a.s., for any $\beta \in [0, \min(1/2, 1/2 + \gamma - d/4))$, and so the sum \eqref{eq:mild-form} has trajectories in $C^{\beta}([t_0, T] ; H)$, $P$-a.s.
\end{proof}

\subsection{Pathwise convergence rates of a fully discrete finite element approximation of (\ref{eq:example-model})}

To approximate the mild solution $u$, we use a discretization similar to that proposed in \cite{2025-auestad}. To that end, let $V_h \subseteq V$, with $N_h := \mathrm{dim}(V_h)$, consist of piecewise first order polynomials defined on a regular collection of simplices with maximum diameter $h$. We make the following assumption on $V_h$ and the finite element approximation $A_{j,h}$ of $A_j$.
\begin{assumption}\label{assumption:fem}
\hfill
\begin{description}
    \myitem{(N1)}\label{cond:quasi-uniform} $V_h \subseteq V$ consists of continuous functions that are first order polynomials when restricted to $\tau \in \mathcal{T}_h$, where $\mathcal{T}_h$ is a collection of simplices with disjoint interior, satisfying $\bigcup_{\tau \in \mathcal{T}_h} \tau = \overline{\mathcal{D}}$, and for some $C > 0$,
    \begin{align*}
        C^{-1} h \leq 2 \rho(\tau) \leq \text{diam}(\tau) \leq 2 r(\tau) \leq C h, \quad \tau \in \mathcal{T}_h,
    \end{align*}
    where $\rho, r$ are the radii of the incircle and circumcircle, respectively. \medskip
    \myitem{(N2)}\label{cond:discrete-operator} The finite dimensional operators $A_{j,h} : V_h \to V_h$ are defined by
    \begin{align*}
        (A_{j,h} u, v)_H = a_j(u,v), \quad \text{for any } u, v \in V_h.
    \end{align*}
    \myitem{(N3)}\label{cond:solution-operator} With $\lambda \geq 0$ as in \ref{cond:continuity-coercivity} and $\pi_h$ the $H$-orthogonal projection onto $V_h$,
    \begin{align*}
        \Vert ((\lambda + A_1)^{-1} - (\lambda + A_{1,h})^{-1} \pi_h) (\lambda + A_1)^{\alpha} \Vert_{L(H)} \leq C h^{2-2\alpha}, 
    \end{align*}
    and 
    \begin{align*}
        \Vert A_2^{\alpha} (A_2^{-1} - A_{2,h}^{-1} \pi_h) \Vert_{L(H)} \leq C h^{2 - 2\alpha},
    \end{align*}
    for $\alpha \in \{0, 1/2\}$, and some $C > 0$. \medskip
    \myitem{(N4)}\label{cond:fem-bound} There is $C > 0$ such that $\Vert (\lambda + A_{1,h})^{-1/2} \pi_h (\lambda + A_1)^{1/2} \Vert_{L(H)} \leq C$.
\end{description}
\end{assumption}
\begin{remark}
    Sufficient conditions for \ref{cond:solution-operator} and \ref{cond:fem-bound} to hold are given in detail in Remark 2.8 in \cite{2025-auestad}. They may be verified under appropriate smoothness, positivity, boundedness and symmetry conditions on the coefficients of $a_j$.
\end{remark} 

Further, let $t_n := n \Delta t$ where $\Delta t := T / N$ for some integer $N > 0$, and $b_{\Delta t}^n$ be an approximation to $b(t_n)$, made precise in condition \ref{cond:b-process-approximation} below.
\begin{description}
    \myitem{(N5)}\label{cond:b-process-approximation} For some $\beta \in [0,1]$, $b$ has $\beta$-Hölder continuous trajectories and $\Vert b \Vert_{C^{\beta}}$ is $\mathcal{F}_T$-measurable. Further, $b_{\Delta t}^n$ is $\mathcal{F}_{t_n}$-measurable, and
    \begin{align*}
        \sup_{n = 0, \dots, N} \vert b(t_n) - b_{\Delta t}^n \vert \leq K \Delta t^{\beta}, \quad P\text{-a.s.},
    \end{align*}
    for some random variable $K > 0$ that does not depend on $N$.
\end{description}
\begin{remark}
    The condition \ref{cond:b-process-approximation} holds for example for any $\beta < 1 /2$ if $b$ is a (predictable) Brownian motion, (or an appropriate transformation thereof, preserving the predictability and Hölder regularity), and $b_{\Delta t}^n = b(t_n)$. Moreover, it holds for any $\beta < 1$ when $b = g \circ f$ with $f = (I - \partial_t^2)^{-1} \mathcal{W}$, where $\mathcal{W}$ is white noise on $L^2(0,1; \mathbb{R})$ independent of $W$, $g \in C^1([0,T] ; \mathbb{R})$, and $b_{\Delta t}^n$ are related to an appropriate finite element approximation of $b$---see Section \ref{section:numerical-experiments} for the details. 
\end{remark}

In order to approximate the fractional power operator $A_2^{-\gamma}$, we use the quadrature in \cite{2019-bonito}. To that end, let
\begin{align}\label{eq:quadrature}
    Q_k^{-\gamma}(A_2) := \frac{k \sin(\pi \gamma)}{\pi} \sum_{j = -M}^N e^{(1-\gamma) y_j } (e^{y_j} I + A_2)^{-1},
\end{align}
where, $k > 0$ is the quadrature resolution, $y_j = j k, \ j = -M, \dots, N$, and
\begin{align*}
    N = \bigg\lceil \frac{\pi^2}{2 \gamma k^2} \bigg\rceil, \quad M = \bigg\lceil \frac{\pi^2 }{2(1 - \gamma) k^2} \bigg\rceil.
\end{align*}
With this choice of $N$ and $M$, the quadrature converges exponentially in $k$, as described in Lemma \ref{lemma:A_h-quadrature}. 

Our fully discrete approximation of \eqref{eq:mild-form} is based on the semidiscrete approximation
\begin{align}\label{eq:semidiscrete-gamma-in-0-1}
    d u_h = -A_{1,h} u_h \, dt + b Q_k^{-\gamma}(A_{2,h}) \pi_h \, dW, \quad u_h(0) = \pi_h \xi,
\end{align}
with convention $Q_k^{-1}(A_{2,h}) = A_{2,h}^{-1}$ and $Q_k^0(A_{2,h}) = I$. Discretizing \eqref{eq:semidiscrete-gamma-in-0-1} in time with backward Euler, and using the approximation $b_{\Delta t}^n$ of $b(t_n)$, we get our fully discrete approximation,
\begin{align}\label{eq:fully-discrete-scheme}
    (I + \Delta t A_{1,h})u_{h,\Delta t}(t_{n+1}) = u_{h,\Delta t}(t_n) + b_{\Delta t}^n Q_k^{-\gamma}(A_{2,h}) \pi_h (W(t_{n+1}) - W(t_n)),
\end{align}
and we define $u_{h,\Delta t}(t), t \in (t_n, t_{n+1})$ by \eqref{eq:fully-discrete-2}. 

We may rewrite \eqref{eq:fully-discrete-scheme} as a system of equations for the coefficients of $u_{h,\Delta t}$ in the nodal basis of $V_h$, which we denote by $\varphi_j, \ j = 1, \dots, N_h$. To that end, let
\begin{align}\label{eq:fem-matrices}
    (M_h)_{ij} = \int_{\mathcal D} \varphi_j \varphi_i \, dx, \quad (T_{h})_{ij} = a_{1,h}(\varphi_j, \varphi_i), \quad (K_h)_{ij} = a_{2,h}(\varphi_j, \varphi_i).
\end{align}
Expressing $u_{h,\Delta t}(t_n) = \sum_{j = 1}^{N_h} \alpha_j^n \varphi_j$, for coefficients $\alpha_j^n$, (see Appendix B in \cite{2025-auestad} for a detailed derivation in the case of $b = 1$), \eqref{eq:fully-discrete-scheme} may be rewritten
\begin{align}\label{eq:scheme-basis-coefficients}
\begin{split}
    &(M_h + \Delta t T_h) \alpha^{n+1} \\
    &\qquad =
    \begin{cases}
        M_h \alpha^n + b_{\Delta t}^n M_h \Delta t^{1/2} \frac{k \sin(\pi \gamma)}{\pi} \sum_{j = -M}^N e^{(1-\gamma) y_j} (e^{y_j} M_h + K_h)^{-1} M_h^{1/2} \varrho_h^n, \quad &\gamma \in (0,1) \\
        M_h \alpha^n + b_{\Delta t}^n M_h \Delta t^{1/2} K_{h}^{-1} M_h^{1/2} \varrho_h^n, \quad &\gamma = 1,
    \end{cases}
\end{split}
\end{align}
where $\varrho_h^n \sim \mathcal{N}(0,I)$ are $N_h$-dimensional multivariate Gaussian and independent for each $n$. 

Before we are able to derive rates of pathwise convergence of our numerical approximation, we need a couple of estimates. To that end, it will be convenient to define $S_{h,\Delta t}(\cdot)$ as our fully discrete approximation of $S_1(\cdot)$ based on backward Euler, 
\begin{align}\label{eq:fully-discrete-semigroup-operator}
    S_{h,\Delta t}(t) :=
    \begin{cases}
        I, \quad &t = 0, \\
        r(\Delta t A_{1,h} )^{n+1}, \quad &t \in (t_n,t_{n+1}],
    \end{cases}
\end{align}
where $r(z) := (1 + z)^{-1}$, and
\begin{align*}
    b_{\Delta t}(t) := 
    \begin{cases}
        b_{\Delta t}^0, \quad &t = 0, \\
        b_{\Delta t}^n, \quad &t \in (t_n, t_{n+1}].    
    \end{cases}
\end{align*}
With these definitions at hand, our approximate mild solution \eqref{eq:fully-discrete-scheme} can be extended from discrete times to all times $t \geq 0$ as
\begin{align}\label{eq:fully-discrete-2}
    u_{h,\Delta t}(t) := S_{h,\Delta t}(t) \pi_h \xi + \int_0^t S_{h,\Delta t}(t-s) b_{\Delta t}(s) Q_k^{\gamma}(A_{2,h}) \pi_h \, dW.
\end{align}
The following lemma lists some useful estimates related to the fully discerete approximation \eqref{eq:fully-discrete-semigroup-operator} of $S_1(t)$. 
\begin{lemma}\label{lemma:error-stochastic-convolution}
    Suppose Assumption \ref{assumption:model} and \ref{assumption:fem} holds. Then for any $\theta \in [0, 2\gamma + 1 - d / 2) \cap [0,2]$, there is $C_{\theta}, c > 0$ such that
    \begin{align*}
        \int_0^t \Vert (S_1(t-s) - S_{h,\Delta t}(t-s) \pi_h) A_2^{-\gamma} \Vert_{L_2(H)}^2 \, ds \leq C_{\theta} e^{2 c \lambda t} (h^{\theta} + \Delta t^{\theta / 2})^2,
    \end{align*}
    and
    \begin{align*}
        \int_0^t \Vert S_{h,\Delta t}(t-s) \pi_h (A_2^{-\gamma} - A_{2,h}^{-\gamma} \pi_h) \Vert_{L_2(H)}^2 \, ds \leq C_{\theta} e^{2 c \lambda t} (h^{\theta} + \Delta t^{\theta / 2})^2.
    \end{align*}
    Moreover, there is $C > 0$ such that
    \begin{align*}
        \int_0^t \Vert S_{h,\Delta t}(t-s) (A_{2,h}^{-\gamma} - Q_k^{-\gamma}(A_{2,h})) \pi_h \Vert_{L_2(H)}^2 \, ds \leq C e^{2 c \lambda t} e^{-\pi^2 / k} N_h.
    \end{align*}
\end{lemma}
\begin{proof}
    For the first two inequalities, see Lemma 4.4 in \cite{2025-auestad} (in particular the treatment of the terms $(i)$ and $(ii)$, respectively). For the last inequality, see the proof of Theorem 3.1 in \cite{2025-auestad} (in particular the treatment of the term $(iii)$). 
\end{proof}

By combining the following lemma with Theorem \ref{theorem:extension-inequality-2} and Lemma \ref{lemma:pathwise}, we can derive pathwise convergence rates for our approximation \eqref{eq:fully-discrete-scheme}.
\begin{lemma}\label{lemma:strong-rate}
    Suppose 
    \begin{enumerate}[(1)]
        \item Assumption \ref{assumption:model}, \ref{assumption:fem} and \ref{cond:b-process-approximation} holds,
        \item $\gamma \in (d/4 - 1/2, 1] \cap [0,1]$,
        \item $k \leq -\frac{\pi^2}{2}(2 \gamma + 1)^{-1} \log(h)^{-1}$,
    \end{enumerate}
    and let $u$ and $u_{h,\Delta t}$ be the solutions to \eqref{eq:mild-form} and \eqref{eq:fully-discrete-2} with $\xi = 0$, respectively. Then,
    \begin{align*}
        u(t) - u_{h,\Delta t}(t) = \int_0^t \Phi \, dW, \quad \Phi := S_1(t-\cdot) b A_2^{-\gamma} - S_{h,\Delta t}(t - \cdot) b_{\Delta t} Q_k^{-\gamma}(A_{2,h}),
    \end{align*}
    where for any $\theta \in [0, 2\gamma + 1 - d / 2) \cap [0,2]$, there is $C_{\theta}, C, c > 0$ such that
    \begin{align*}
        \int_0^t \Vert \Phi \Vert_{L_2(H)}^2 \, ds \leq C_{\theta} (\Vert b \Vert_{\infty} + K)^2 e^{2c\lambda t} (h^{\theta} + \Delta t^{\theta / 2})^2 + C (\Vert b \Vert_{C^{\beta}} + K)^2 e^{2c \lambda t} \Delta t^{2\beta}.
    \end{align*}
\end{lemma}
\begin{proof}
    Note first that owing to \ref{cond:b-process-approximation}, we have
    \begin{align*}
        \Vert b - b_{\Delta t} \Vert_{\infty} \leq \Vert b \Vert_{C^{\beta}} \Delta t^{\beta} + \sup_{n = 0, \dots, N} \vert b(t_n) - b_{\Delta t}^n \vert \leq (\Vert b \Vert_{C^{\beta}} + K) \Delta t^{\beta},
    \end{align*}
    and $\Vert b_{\Delta t} \Vert_{\infty} \leq \Vert b \Vert_{\infty} + K \Delta t^{\beta}$, $P$-a.s.

    For $s \in [0,t]$
    \begin{align*}
        \Phi(s) &= S_1(t-s) (b(s) - b_{\Delta t}(s)) A_2^{-\gamma} \\
        &\quad + (S_1(t-s) - S_{h,\Delta t}(t-s) \pi_h) b_{\Delta t}(s) A_2^{-\gamma} \\
        &\quad + S_{h,\Delta t}(t-s) \pi_h b_{\Delta t}(s) (A_2^{-\gamma} - A_{2,h}^{-\gamma} \pi_h) \\
        &\quad + S_{h,\Delta t}(t-s) b_{\Delta t}(s) (A_{2,h}^{-\gamma} - Q_k^{-\gamma}(A_{2,h})) \pi_h) \\
        &=: (i) + (ii) + (iii) + (iv),
    \end{align*}
    and by Young's inequality
    \begin{align*}
        \int_0^t \Vert \Phi \Vert_{L_2(H)}^2 \, ds \leq C \int_0^t (\Vert (i) \Vert_{L_2(H)}^2 + \cdots + \Vert (iv) \Vert_{L_2(H)}^2) \, ds.
    \end{align*}

    For $(i)$, we get by \ref{cond:hilbert-schmidt} that for $0 < \epsilon < \gamma - d / 4 + 1 / 2$ 
    \begin{align*}
        \int_0^t \Vert (i) \Vert_{L_2(H)}^2 \, ds &= \int_0^t \Vert S_1(t-s) (b(s) - b_{\Delta t}(s)) A_2^{-\gamma} \Vert_{L_2(H)}^2 \, ds \\
        &\leq \int_0^t \Vert S_1(t-s) A_2^{1/2 - \epsilon} \Vert_{L(H)}^2 \Vert b - b_{\Delta t} \Vert_{\infty}^2 \Vert A_2^{-\gamma - 1 / 2 + \epsilon} \Vert_{L_2(H)}^2 \, ds \\
        &\leq C_{\epsilon} e^{2 \lambda t} \Vert b - b_{\Delta t} \Vert_{\infty}^2,
    \end{align*}
    where we used the estimate $\Vert S_1(t-s) A_2^{\alpha} \Vert_{L(H)} \leq C e^{(\lambda - \delta) (t-s)} (t-s)^{-\alpha}$, $\alpha \in [0,1/2]$ (which follows by Lemma \ref{lemma:analytic-semigroup}, condition \ref{cond:A1A2}, and interpolation) for the first factor in the integrand.
    
    For $(ii)$ Lemma \ref{lemma:error-stochastic-convolution} gives
    \begin{align*}
        \int_0^t \Vert (ii) \Vert_{L_2(H)}^2 \, ds &= \int_0^t \Vert (S_1(t-s) - S_{h,\Delta t}(t-s) \pi_h) b_{\Delta t}(s) A_2^{-\gamma} \Vert_{L_2(H)}^2 \, ds \\
        &\leq \Vert b_{\Delta t} \Vert_{\infty}^2 \int_0^t \Vert (S_1(t-s) - S_{h,\Delta t}(t-s) \pi_h) A_2^{-\gamma} \Vert_{L_2(H)}^2 \, ds \\
        &\leq C_{\theta} \Vert b_{\Delta t} \Vert_{\infty}^2 e^{2 c \lambda t} (h^{\theta} + \Delta t^{\theta / 2})^2,
    \end{align*}
    where we used that scaling by $b_{\Delta t}(s)$ viewed as an operator $H \to H$ has operator norm bounded by $\Vert b_{\Delta t} \Vert_{\infty}$.
    
    For $(iii)$, we get by using Lemma \ref{lemma:error-stochastic-convolution}
    \begin{align*}
        \int_0^t \Vert (iii) \Vert_{L_2(H)}^2 \, ds &= \int_0^t \Vert S_{h,\Delta t}(t-s) \pi_h b_{\Delta t}(s) (A_2^{-\gamma} - A_{2,h}^{-\gamma} \pi_h) \Vert_{L_2(H)}^2 \, ds \\
        &\leq \Vert b_{\Delta t} \Vert_{\infty}^2 \int_0^t \Vert S_{h,\Delta t}(t-s) \pi_h (A_2^{-\gamma} - A_{2,h}^{-\gamma} \pi_h) \Vert_{L_2(H)}^2 \, ds \\
        &\leq C_{\theta} \Vert b_{\Delta t} \Vert_{\infty}^2 e^{2 c \lambda t} (h^{\theta} + \Delta t^{\theta / 2})^2.
    \end{align*}
    
    Finally, for $(iv)$, we have using Lemma \ref{lemma:error-stochastic-convolution}
    \begin{align*}
        \int_0^t \Vert (iv) \Vert_{L_2(H)}^2 \, ds &= \int_0^t \Vert S_{h,\Delta t}(t-s) b_{\Delta t}(s) (A_{2,h}^{-\gamma} - Q_k^{-\gamma}(A_{2,h})) \pi_h \Vert_{L_2(H)}^2 \\
        &\leq \Vert b_{\Delta t} \Vert_{\infty}^2 \int_0^t \Vert S_{h,\Delta t}(t-s) (A_{2,h}^{-\gamma} - Q_k^{-\gamma}(A_{2,h})) \pi_h \Vert_{L_2(H)}^2 \\
        &\leq C \Vert b_{\Delta t} \Vert_{\infty}^2 e^{2 c \lambda t} e^{-\pi^2 / k} N_h.
    \end{align*} 
    Using the bound on $k$ in terms of $h$, the bound $N_h \leq C h^{-d}$ which follows from \ref{cond:quasi-uniform}, and Lemma \ref{lemma:error-stochastic-convolution}, this term is bounded by $C \Vert b_{\Delta t} \Vert_{\infty}^2 e^{2c \lambda t} (h^{2\gamma + 1 - d/2})^2$. If $\gamma = 1$, this term vanishes. By combining the four terms and using the estimates of $\Vert b - b_{\Delta t} \Vert_{\infty}$ and $\Vert b_{\Delta t} \Vert_{\infty}$ due to \ref{cond:b-process-approximation}, we get the estimate. 
\end{proof}

We are now ready to derive pathwise convergence rates for our numerical approximation \eqref{eq:fully-discrete-scheme}.
\begin{theorem}\label{theorem:pathwise-rate}
    Suppose the conditions of Lemma \ref{lemma:strong-rate} holds. Then, for any $\theta \in [0, 2\gamma + 1 - d/2) \cap [0,2]$, $\rho \in [-1,\theta] \cap [-2 + \theta, \theta]$, $\epsilon > 0$ small, and sequences $h_n, \Delta t_n$ such that
    \begin{align*}
        \sum_{n = 1}^{\infty} (h_n^{\theta} + \Delta t_n^{\frac{\theta}{2} \wedge \beta})^p < \infty,
    \end{align*}
    for some $p > 0$, there is a random variable $M_{\theta, \epsilon} > 0$ and constants $C, c > 0$, ensuring that
    \begin{align*}
        \Vert u(t) - u_{h_n, \Delta t_n}(t) \Vert_H &\leq e^{c \lambda t}\bigg(C t^{-\theta/2 + \rho / 2} \Vert (\lambda + A_1)^{\rho / 2} \xi \Vert_H + M_{\theta, \epsilon} \bigg) \\
        &\quad \times (h_n^{\theta} + \Delta t_n^{(\frac{\theta}{2}\wedge \beta)})^{1-\epsilon}, \quad P\text{-a.s.}
    \end{align*}
\end{theorem}
\begin{proof}
    We can express the error
    \begin{align*}
        u(t) - u_{h,\Delta t}(t) &= S_1(t) \xi - S_{h,\Delta t}(t) \pi_h \xi \\
        &\quad +\int_0^t (S_1(t-s) b(s) A_2^{-\gamma} \, dW - S_{h,\Delta t}(t-s) b_{\Delta t}(s) Q_k^{-\gamma}(A_{2,h}) \pi_h \, dW) \\
        &= (i) + (ii).
    \end{align*}
    By Lemma \ref{lemma:fully-discrete-semigroup-approximation} we have
    \begin{align*}
        \Vert (i) \Vert_H \leq C e^{c \lambda t} t^{-\theta/2 + \rho/2} (h^{\theta} + \Delta t^{\theta/2}) \Vert (\lambda + A_1)^{\rho / 2} \xi \Vert_H, \quad P\text{-a.s.}
    \end{align*}
    for any $\theta \in [0,2]$ and $\rho \in [-1, \theta] \cap [-2 + \theta, \theta]$. 
    
    For $(ii)$, note first that by Lemma \ref{lemma:strong-rate}, we must have for $\theta \in [0, 2\gamma + 1 - d / 2) \cap [0,2]$
    \begin{align*}
        &\int_0^t \Vert S_1(t-s) b(s) A_2^{-\gamma} - S_{h,\Delta t}(t-s) b_{\Delta t}(s) Q_k^{-\gamma}(A_{2,h}) \pi_h \Vert_{L_2(H)}^2 \, ds \\
        &\qquad \leq C_{\theta} (\Vert b \Vert_{C^{\beta}} + K)^2 e^{2c\lambda t} (h^{\theta} + \Delta t^{\frac{\theta}{2} \wedge \beta})^2.
    \end{align*}
    Therefore, by Theorem \ref{theorem:extension-inequality-2} it follows that for sequences $h_n, \Delta t_n$ as in this theorem, and for any $\epsilon > 0$, $p' \geq 2$
    \begin{align*}
        E[1 \wedge \sum_{m \geq n} \frac{\Vert (ii) \Vert_H^{p'}}{(h_m^{\theta} + \Delta t_m^{\frac{\theta}{2} \wedge \beta})^{(1-\epsilon)p'}}] \leq C_{p'} E[1 \wedge C_{\theta} (\Vert b \Vert_{C^{\beta}} + K)^{p'} e^{p' c \lambda t} \sum_{m \geq n} (h_m^{\theta} + \Delta t_m^{\frac{\theta}{2} \wedge \beta})^{p' \epsilon}] \to 0,
    \end{align*}
    as $n \to \infty$, provided we choose $p' \geq p \epsilon^{-1}$. Therefore, by Lemma \ref{lemma:pathwise}, there is a random variable $M_{\theta,\epsilon} > 0$, satisfying
    \begin{align*}
        \Vert (ii) \Vert_H \leq M_{\theta,\epsilon} e^{c\lambda t} (h_n^{\theta} + \Delta t_n^{\frac{\theta}{2} \wedge \beta})^{1 - \epsilon}, \quad P\text{-a.s.}
    \end{align*}
    By combining the two terms, we get the inequality of this theorem. 
\end{proof}

\subsection{Numerical experiments}\label{section:numerical-experiments}

We now consider the specific case $A_1 = -\Delta$, $A_2 = I - \Delta$, where $\Delta$ is the Laplacian with zero Neumann boundary conditions, and $\xi = 0$. Further, we let $b := e^{f^2}$, where $f := (I - \partial_t^2)^{-1} \mathcal{W}$, $\partial_t^2$ is the (one dimensional) Laplacian with zero Neumann boundary conditions, and $\mathcal{W}$ is white noise on $L^2(0,T)$, independent of $W$. Moreover, we let $b_{\Delta t}^n := e^{(f_{\Delta t}^n)^2}$ where $f_{\Delta t}^n$ are the nodal basis coefficients of a finite element approximation of $f$ with a uniform mesh of size $\Delta t$. In this case, one may verify that Assumption \ref{assumption:model} and \ref{assumption:fem} holds (see, e.g., Remark 2.3 and 2.8 in \cite{2025-auestad}). In particular, condition \ref{cond:continuity-coercivity} holds with any $\lambda > 0$. Moreover, condition \ref{cond:b-process-approximation} holds for any $\beta < 1$. 

To see that \ref{cond:b-process-approximation} holds, note first that since $e^{(\cdot)^2} \in C^1([0,T])$ it suffices to verify that the Hölder continuity and estimate of \ref{cond:b-process-approximation} holds for the process $f$ and coefficients $f_{\Delta t}^n$. We start by verifying the Hölder continuity of $f$. We assume for simplicity and without loss of generality that $T = 1$, define $e_n(t) := 2^{1/2} \cos(\pi n t), n > 0$, $e_0(t) := 1$, and note that
\begin{align}\label{eq:example-f}
    f(t) = \sum_{n = 0}^{\infty} (1 + \pi^2 n^2)^{-1} \xi_n e_n(t), \quad P\text{-a.s.}
\end{align}
for some independent $\xi_n \sim \mathcal{N}(0,1)$. By using that $\vert e_n(t) - e_n(s) \vert \leq 2^{1/2} \pi n \vert t - s \vert$, one finds
\begin{align*}
    E[\vert f(t) - f(s) \vert^2] = E[ ( \sum_{n = 0}^{\infty} (1 + \pi^2 n^2)^{-1} \xi_n (e_n(t) - e_m(t) )^2 ] \leq C \vert t - s \vert^2,
\end{align*}
and so the $\beta$-Hölder continuity of $f$ follows by the Kolmogorov continuity test for Gaussian processes. (Similarly, using the expression $\dot{f}(t) = -\sum_n (1 + \pi^2 n^2)^{-1} \pi n \xi_n 2^{1/2} \sin(\pi n t)$ and inequality $\vert \sin(\pi n t) - \sin(\pi n s) \vert \leq C n^{\alpha} \vert t  - s \vert^{\alpha}$ for $\alpha \in [0,1]$, one may verify that $f \in C^{1,\beta}([0,T])$ for any $\beta < 1/2$, $P$-a.s.).

To show the estimate of \ref{cond:b-process-approximation}, let $B := I-\partial_t^2$, defined in the usual sense by a coercive and continuous bilinear form defined on the subspace of $H^1(0,1)$ with zero Neumann boundary conditions, and let $B_{\Delta t}$ be defined analogously as in \ref{cond:discrete-operator}, for a finite elements space $V_{\Delta t}$ having uniform mesh with size $\Delta t$, and $L^2(0,1)$-orthogonal projection $\pi_{\Delta t} : L^2(0,1) \to V_{\Delta t}$. In this case, one may verify that (see, e.g., Chapter 7.3 in \cite{yagi})
\begin{align}\label{eq:numerical-example-condition-1}
    \Vert B^{\alpha / 2} (B^{-1} - B_{\Delta t}^{-1} \pi_{\Delta t}) \Vert_{L(L^2)} \leq C \Delta t^{2 - \alpha}, \quad \Vert (I - \pi_{\Delta t}) B^{-\alpha} \Vert_{L(L^2)} \leq C \Delta t^{2\alpha},
\end{align}
for $\alpha \in [0,1]$, while
\begin{align}\label{eq:numerical-example-condition-2}
    \Vert B^{-\alpha} \Vert_{L_2(L^2)} \leq C_{\alpha},
\end{align}
for any $\alpha > 1 / 4$, and finally
\begin{align}\label{eq:numerical-example-condition-3}
    \Vert \pi_{\Delta t} \Vert_{L_2(L^2)} \leq C \Delta t^{-1/2}.
\end{align}
Further, for any $\epsilon > 0$ and $f \in D(B^{1/4 + \epsilon})$, with $f = \sum_n \alpha_n e_n$, we have $\Vert B^{1/4 + \epsilon} f \Vert_{L^2} = \sum_n \alpha_n^2 (1 + \pi^2 n^2)^{1/2 + 2\epsilon}$, and
\begin{align*}
    \Vert f \Vert_{\infty} \leq 2^{1/2}\sum_{n = 0}^{\infty} \vert \alpha_n \vert = 2^{1/2}\sum_{n = 0}^{\infty} \vert \alpha_n \vert n^{1 / 2 + 2\epsilon} n^{-1 / 2 - 2\epsilon} \leq C_{\epsilon} (\sum_{n = 0}^{\infty} \alpha_n^2 n^{1 + 4 \epsilon})^{1/2} \leq C_{\epsilon} \Vert B^{1/4 + \epsilon} f \Vert_{L^2},
\end{align*}
by Cauchy--Schwarz. Therefore, we have, for any $\epsilon > 0$
\begin{align*}
    \sup_{n = 0, \dots, N} \vert f(t_n) - f_{\Delta t}^n \vert \leq \Vert (B^{-1} - B_{\Delta t}^{-1} \pi_{\Delta t}) \mathcal{W} \Vert_{\infty} \leq C_{\epsilon} \Vert B^{1/4 + \epsilon} (B^{-1} - B_{\Delta t}^{-1} \pi_{\Delta t}) \mathcal{W} \Vert_{L^2}.
\end{align*}
Hence, by Lemma \ref{lemma:pathwise} it suffices to show that for $\Delta t_m = m^{-1}$ (recall $T = 1$ for simplicity) and some $p > 0$
\begin{align*}
    E[1 \wedge \sum_{m \geq n} \frac{\Vert B^{1/4 + \epsilon} (B^{-1} - B_{\Delta t_m}^{-1} \pi_{\Delta t_m}) \mathcal{W} \Vert_{L^2}^p}{\Delta t_m^{\beta p}}] \to 0, \quad \text{as } n \to \infty.
\end{align*}
Observing that for any $A \in L_2(L^2(0,1))$, the law of $A \mathcal{W}$ is the same as that of $\int_0^1 A \, dW$ for a cylindrical Wiener process $W$ on $L^2(0,1)$, Theorem \ref{theorem:extension-inequality-2} gives
\begin{align*}
    &E[1 \wedge \sum_{m \geq n} \frac{\Vert B^{1/4 + \epsilon} (B^{-1} - B_{\Delta t_m}^{-1} \pi_{\Delta t_m}) \mathcal{W} \Vert_{L^2}^p}{\Delta t_m^{\beta p}}] \\
    &\qquad \leq C_p E[1 \wedge \sum_{m \geq n} \frac{\Vert B^{1/4 + \epsilon}(B^{-1} - B_{\Delta t_m}^{-1} \pi_{\Delta t_m}) \Vert_{L_2(L^2)}^p}{\Delta t_m^{\beta p}}],
\end{align*}
for any $p \geq 2$, where
\begin{align*}
    \Vert B^{1/4 + \epsilon} (B^{-1} - B_{\Delta t}^{-1} \pi_{\Delta t}) \Vert_{L_2(L^2)} &= \Vert (I - \pi_{\Delta t}) B^{-3/4 + \epsilon} + \pi_{\Delta t} B^{1/4 + \epsilon} (B^{-1} - B_{\Delta t}^{-1} \pi_{\Delta t}) \Vert_{L_2(L^2)} \\
    &\leq \Vert (I - \pi_{\Delta t}) B^{-3/4 + \epsilon} \Vert_{L_2(L^2)} \\
    &\quad + \Vert \pi_{\Delta t} B^{1/4 + \epsilon} (B^{-1} - B_{\Delta t}^{-1} \pi_{\Delta t}) \Vert_{L_2(L^2)} \\
    &=: \Vert (i) \Vert_{L_2(L^2)} + \Vert (ii) \Vert_{L_2(L^2)}.
\end{align*}
We have by \eqref{eq:numerical-example-condition-1}, \eqref{eq:numerical-example-condition-2} and \eqref{eq:numerical-example-condition-3}
\begin{align*}
    \Vert (i) \Vert_{L_2(L^2)} \leq \Vert (I - \pi_{\Delta t}) B^{-1/2 + 2\epsilon} \Vert_{L(L^2)} \Vert B^{-1/4 - \epsilon} \Vert_{L_2(L^2)} \leq C_{\epsilon} \Delta t^{1 - 4\epsilon},
\end{align*}
while
\begin{align*}
    \Vert (ii) \Vert_{L_2(L^2)} \leq \Vert B^{1/4 + \epsilon} (B^{-1} - B_{\Delta t}^{-1} \pi_{\Delta t}) \Vert_{L(L^2)} \Vert \pi_{\Delta t} \Vert_{L_2(L^2)} \leq C \Delta t^{1 - 2 \epsilon}. 
\end{align*}
Hence,
\begin{align*}
    E[1 \wedge \sum_{m \geq n} \frac{\Vert B^{1/4 + \epsilon} (B^{-1} - B_{\Delta t_m} \pi_{\Delta t_m}) \Vert_{L_2(L^2)}^p}{\Delta t_m^{\beta p}}] \leq E[1 \wedge C_{\epsilon} \sum_{m \geq n} \Delta t_m^{(1-4\epsilon - \beta) p}] \to 0,
\end{align*}
for any $\beta < 1 - 4\epsilon$ provided we choose $p = 2 (1-4\epsilon-\beta)^{-1}$, since $\Delta t_m = m^{-1}$. Since $\epsilon > 0$ can be chosen arbitrarily small, the estimate of \ref{cond:b-process-approximation} follows.

Moreover, one easily verifies that $b \notin L^2(\Omega ; L^2(0,1))$, and so $S_1(t - \cdot) b A_2^{-\gamma} \notin L^2(\Omega ; L^2(0,t ; L_2(H)))$, from which we gather that the standard $L^2$-dependent approaches no longer applies in this case. To see this, note that by \eqref{eq:example-f}, we have $E[\vert f(t) \vert^2] \geq 1$. Since $f(t)$ is Gaussian, it has a probability density function which cannot decay faster than $e^{-f(t)^2 / 2}$, ensuring that $e^{f(t)^2} \notin L^2(\Omega ; \mathbb{R})$ for any $t$. 

We approximate the relative pathwise error at time $t = 1$ by
\begin{align}\label{eq:relative-pathwise-error}
    e_{h,\Delta t} := \frac{\Vert u_{h,\Delta t}(1)-u_{\tilde{h},\widetilde{\Delta t}}(1) \Vert_H}{\Vert u_{\tilde{h},\widetilde{\Delta t}}(1)\Vert_H},
\end{align}
where a coarse approximation, $u_{h, \Delta t}(1)$, is compared to a reference solution, $u_{\tilde{h}, \widetilde{\Delta t}}(1)$, based on a finer space and time resolution, $\tilde{h}$, $\widetilde{\Delta t}$. In order to compute $u_{h,\Delta t}$ using the same cylindrical Wiener process $W$ that was used for computing $u_{\tilde{h}, \widetilde{\Delta t}}$ we observe the following: for an appropriate coarser time resolution, $\Delta t$, we can construct increments of $\pi_{\tilde{h}} W$ by summing increments from the reference time resolution. Moreover, for a coarser space resolution where we can express any nodal basis function $\varphi_n = \sum_j \varphi_n(\tilde{x}_j) \tilde{\varphi_j}$, with $\tilde{x}_j$ a vertex of the finer mesh and $\tilde{\varphi}_j$ its nodal basis function, we can compute $\varrho_h^n$ in \eqref{eq:scheme-basis-coefficients} from $\varrho_{\tilde{h}}^{n}$ by
\begin{align*}
    \Delta t^{1/2} M_h^{1/2} \varrho_h^n &= 
    \begin{pmatrix}
        (\varphi_1, \pi_h (W(t_{n+1}) - W(t_n)))_H \\
        \vdots \\
        (\varphi_{N_h}, \pi_h (W(t_{n+1}) - W(t_n)))_H
    \end{pmatrix}
    \\
    &= A 
    \begin{pmatrix}
        (\tilde{\varphi}_1, \pi_{\tilde{h}} (W(t_{n+1}) - W(t_n)))_H \\
        \vdots \\
        (\tilde{\varphi}_{N_{\tilde{h}}}, \pi_{\tilde{h}} (W(t_{n+1}) - W(t_n)))_H
    \end{pmatrix}
    \\
    &= \Delta t^{1/2} A M_{\tilde{h}}^{1/2} \varrho_{\tilde{h}}^n,
\end{align*}
where $A_{ij} = \varphi_i(\tilde{x}_j)$. 

Hence, realizations of $e_{h,\Delta t}$ in \eqref{eq:relative-pathwise-error} may be computed as follows: we simulate a realization of the reference and coarse solution, $u_{\tilde{h},\widetilde{\Delta t}}(1) = \sum_{j = 1}^{N_{\tilde{h}}} \tilde{\alpha}_j \tilde{\varphi}_j$ and $u_{h,\Delta t}(1) = \sum_{j = 1}^{N_h}\alpha_j \varphi_j$. Then,
\begin{align}\label{eq:relative-pathwise-error-approxmiation}
    e_{h,\Delta t}^2 := \frac{ (A^T\alpha - \tilde{\alpha})^T M_{\tilde{h}}(A^T\alpha-\tilde{\alpha})}{ \tilde{\alpha}^T M_{\tilde{h}} \tilde{\alpha}}.
\end{align}

\begin{example}\label{ex:1d-rates}
We consider \eqref{eq:example-model} with $A_1$ and $A_2$ as in this section, where $\mathcal{D} = (0,1)$. Moreover, $k = 0.5$, $\widetilde{\Delta t} = 2^{-20}$ and $\tilde{h} = 2^{-11}$. The first row in Figure \ref{fig:rates} shows the experimental convergence rate in space and time for $\gamma = 0, 0.25, 0.5$ and $0.75$, together with the corresponding theoretical rates from Theorem \ref{theorem:pathwise-rate}.
\end{example}

\begin{example}\label{ex:2d-rates}
We consider \eqref{eq:example-model} with $A_1$ and $A_2$ as in this section, where $\mathcal{D} = (0,1)^2$. Moreover, $k = 0.5$, $\widetilde{\Delta t} = 2^{-19}$ and $\tilde{h} = 2^{-6.5}$. The second row in Figure \ref{fig:rates} shows the experimental convergence rate in space and time for $\gamma = 0.25, 0.5, 0.75$ and $1$, together with the corresponding theoretical rates from Theorem \ref{theorem:pathwise-rate}.
\end{example}

In Example \ref{ex:1d-rates} and \ref{ex:2d-rates} computations are sped up by taking advantage of the fact that $(1-\Delta)^{-\gamma}$ and $\Delta$ commute so that $(1-\Delta)^{-\gamma}$ only needs to be applied at the final time.

\begin{figure}
    \centering
    \subcaptionbox{Dashed lines show rates $\frac{1}{2}, 1, \frac{3}{2}, 2$.}{
        \includegraphics[width = 6.5cm]{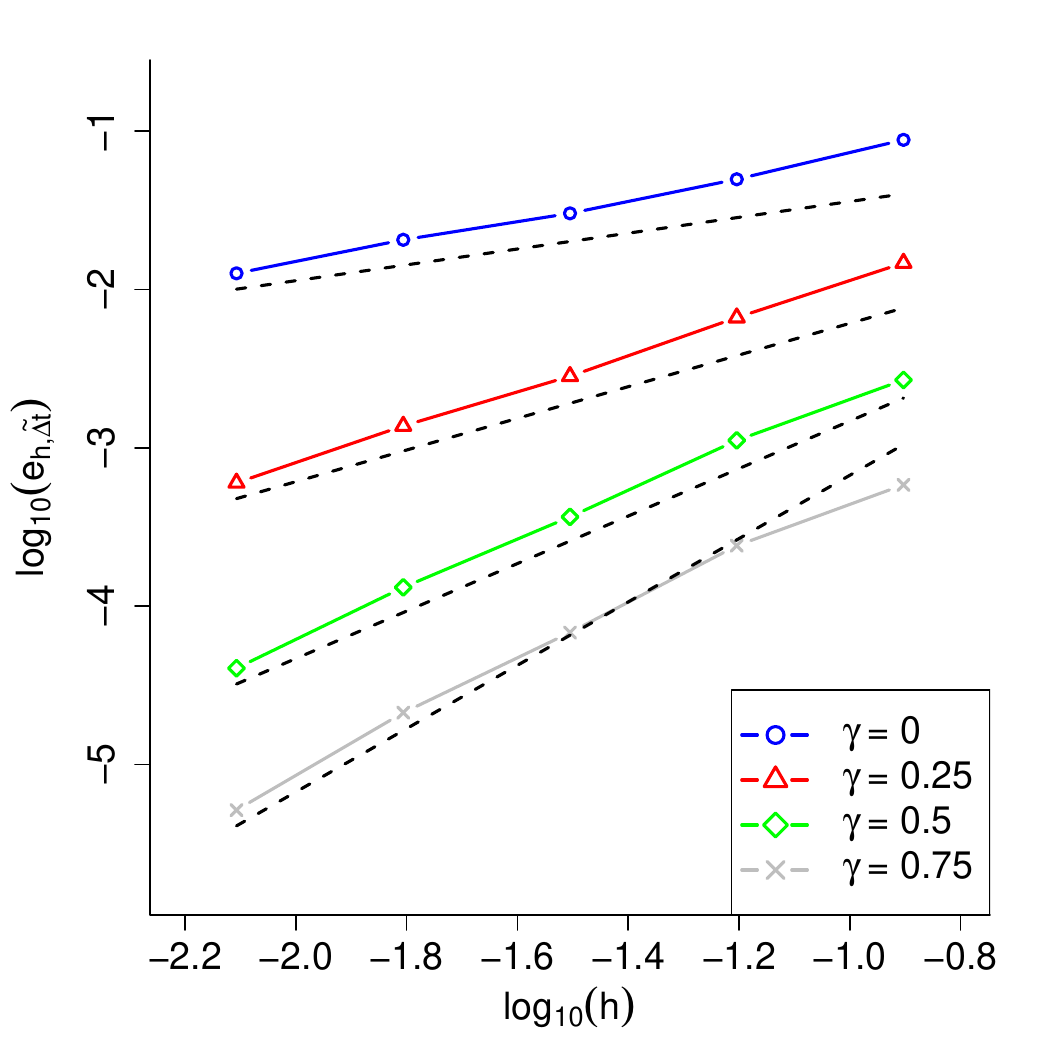}
    }
    \subcaptionbox{Dashed lines show rates $\frac{1}{4},\frac{1}{2},\frac{3}{4}, 1$}{
        \includegraphics[width = 6.5cm]{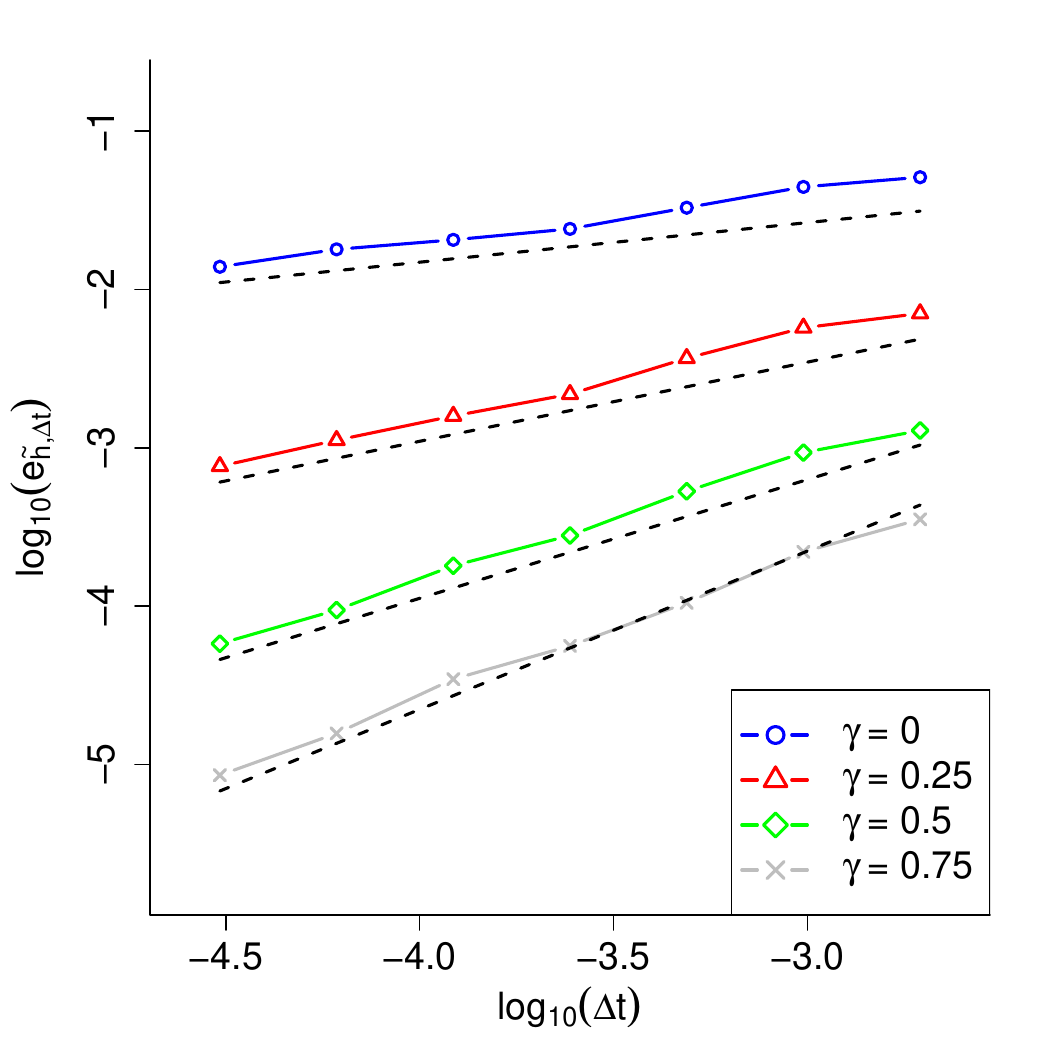}
    } \\
    \subcaptionbox{Dashed lines show rates $\frac{1}{2},1,\frac{3}{2},2$}{
        \includegraphics[width = 6.5cm]{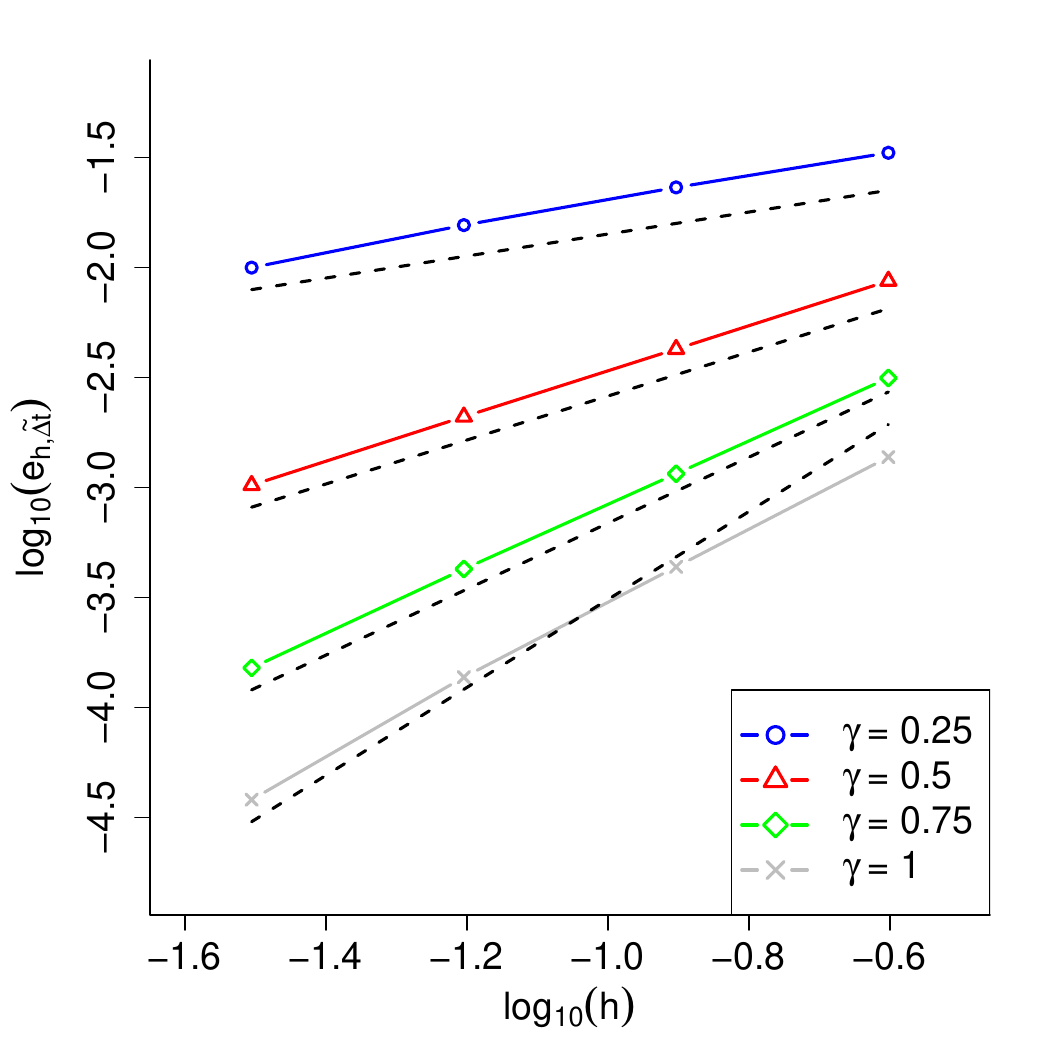}
    }
    \subcaptionbox{Dashed lines show rates $\frac{1}{4},\frac{1}{2},\frac{3}{4},1$.}{
        \includegraphics[width = 6.5cm]{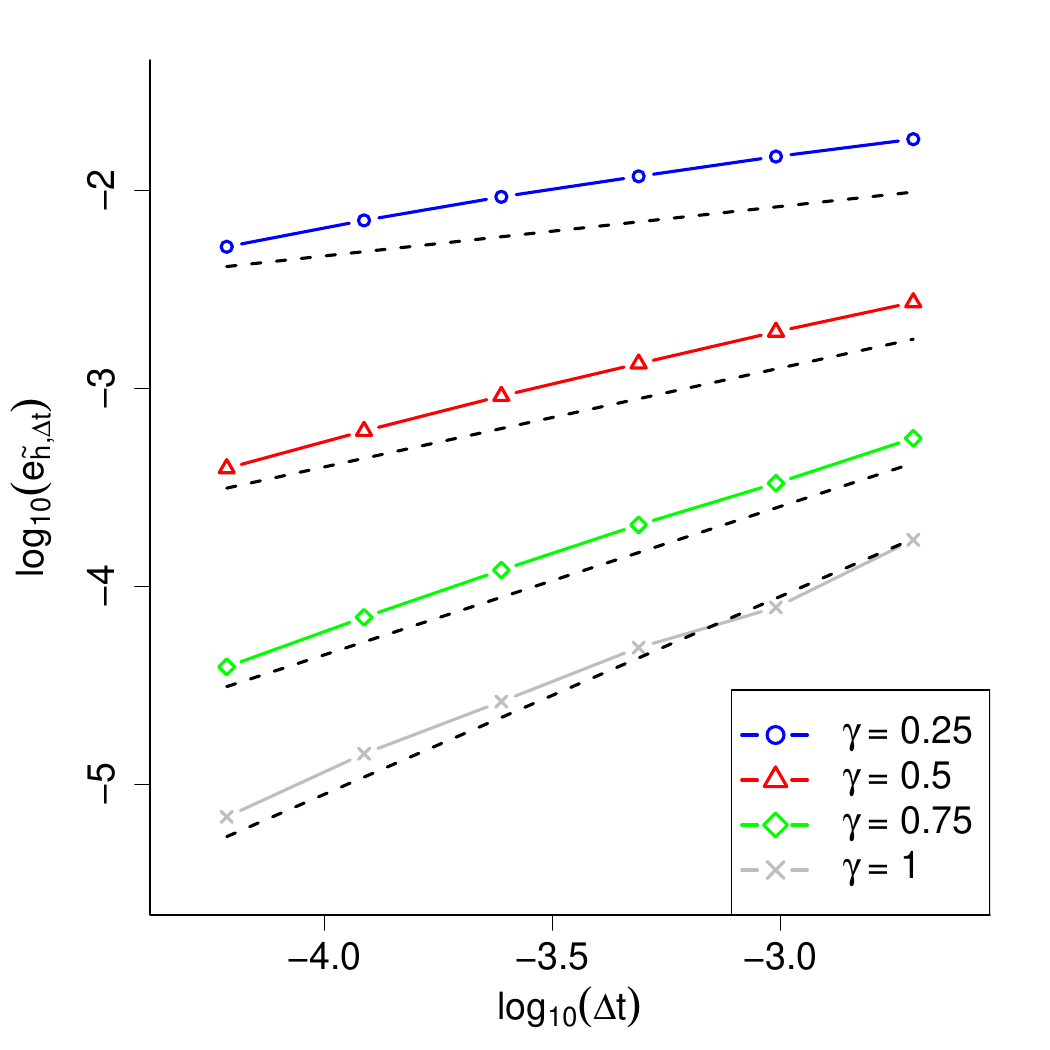}
    }
    \caption{Relative pathwise errors. Row 1 and 2 corresponds to Example 1 and 2, while column 1 and 2 show rates in space and time, respectively. The dashed lines show corresponding theoretical rates. \label{fig:rates}} 
\end{figure}

%% file: appendix-A.tex
\section{}\label{app:A}

\begin{proof}[Proof of lemma \ref{lemma:metric}]
That $d_p$ is a metric follows from the Minkowski inequality for integrals.

First we show that $x_n \xrightarrow{P} x \implies d_{p}(x_n,x)\to 0$. Since $x_n$ converges to $x$ in probability, there exists a subsequence, $\{ x_{n_k} \}_k$ converging to $x$ almost surely. Hence, $1 \wedge d(x_{n_k},x)^p$ converges to $0$ in $L^1(\Omega ; \mathbb{R})$ by dominated convergence theorem. Note that for any subsequence $\{ x_{n_k} \}_k$ we can find such a subsequence. Therefore, since convergence in $L^1(\Omega; \mathbb{R})$ is in a topology, the original sequence $1 \wedge d(x_n,x)^p$ has to converge to $0$ in $L^1(\Omega ; \mathbb{R})$, and therefore also in our metric.

For the other direction, $d_p(x_n,x) \to 0 \implies x_n \xrightarrow{P} x$, we note that since $1\wedge d(x_n,x)^p \to 0$ in $L^1(\Omega ; \mathbb{R})$, we must have that it also converges in probability. Therefore also $d(x_n, x)$ converges to $0$ in probability, and the implication follows. 

It remains to show completeness. For this it suffices to show that any metric that gives the topology of convergence in probability is complete. Thus, it suffices to show it only for the metric $d_1(x,y)$. To do so, we use the Cauchy summability characterization of completeness: that is, $(L^0(\Omega ; X), d_1)$ is complete if and only if it has the property that for every sequence $\{x_n \}_n \subseteq L^0(\Omega ; X)$, where,
\begin{align*}
    \sum_{n = 1}^{\infty} d_1(x_n,x_{n+1}) < \infty,
\end{align*}
we must have $\lim_{n \to \infty} x_n \in L^0(\Omega ; X)$. Note that if,
\begin{align*}
    \sum_{n = 1}^{\infty} d_1(x_n,x_{n+1}) = \sum_{n = 1}^{\infty} E[1 \wedge d(x_n, x_{n+1})] = E[\sum_{n = 1}^{\infty} 1 \wedge d(x_n, x_{n+1})] < \infty,
\end{align*}
we must have that $\sum_n d(x_n, x_{n+1}) < \infty$ for almost every $\omega$. Therefore, for almost every $\omega$ we must have that $x_n(\omega) \to x(\omega)$, by the completeness of $X$, and so $x_n \to x$ almost surely, and therefore also converges in probability. It follows that $x \in L^0(\Omega ; X)$, since it is an almost sure limit of measurable functions, and completeness follows. 
\end{proof}

\begin{proof}[Proof of Lemma \ref{lemma:continuity-test}]
    Assume for simplicity $T = 1$. Note that
    \begin{align*}
        &P(\limsup_n \{\max_{j=0, \dots, 2^n-1} \frac{d(x((j+1)2^{-n}), x(j2^{-n}))}{(2^{-n})^{\beta}} > 1 \}) \\
        &\qquad = P(\limsup_n \{\max_{j=0,\dots,2^n-1} \frac{d(x((j+1)2^{-n}), x(j2^{-n}))^{\delta}}{(2^{-n})^{\beta \delta}} > 1 \}) \\
        &\qquad \leq P(\limsup_n \{\sum_{j=0}^{2^n - 1} \frac{d(x((j+1)2^{-n}), x(j2^{-n}))^{\delta}}{(2^{-n})^{\beta \delta}} > 1 \}) \\
        &\qquad \leq P(\{\sum_{m \geq n} \sum_{j=0}^{2^m - 1} \frac{d(x((j+1)2^{-m}), x(j2^{-m}))^{\delta}}{(2^{-m})^{\beta \delta}} > 1 \}) \\
        &\qquad \leq E[1 \wedge \sum_{m \geq n} \sum_{j=0}^{2^m - 1} \frac{d(x((j+1)2^{-m}), x(j2^{-m}))^{\delta}}{(2^{-m})^{\beta \delta}}] \to 0,
    \end{align*}
    as $n \to \infty$, and so the left hand side is $0$. It follows that there is an integer valued random variable $N$, with $N(\omega) < \infty$ $P$-a.s. such that for all $n \geq N(\omega)$, we have
    \begin{align}\label{eq:pathwise-estiamte-1}
        \max_{j=0,\dots,2^n-1} d(x((j+1) 2^{-n}), x(j 2^{-n})) \leq (2^{-n})^{\beta}. 
    \end{align}
    For $n < N(\omega)$, we have
    \begin{align}\label{eq:pathwise-estiamte-2}
        \max_{j=0,\dots,2^n-1} d(x((j+1)2^{-n}), x(j 2^{-n})) \leq 2^{N(\omega)} (2^{-n})^{\beta}. 
    \end{align}
    Note that for any $0 \leq s < t \leq 1$, we have $t' := j 2^{-\lceil -\log_2(t-s) \rceil} \in [s,t]$ for some $j = 0, \dots, 2^{\lceil -\log_2(t-s) \rceil}$, and 
    \begin{align*}
        d(x(t), x(s)) \leq d(x(t), x(t')) + d(x(t'), x(s)),
    \end{align*}
    where we can express $t = t' + \sum_{j = j_0}^{\infty} \epsilon_j 2^{-j}$, $\epsilon_j \in \{ 0, 1 \}$, with $j_0 := -\lceil -\log_2(t-s) \rceil$. Therefore, using \eqref{eq:pathwise-estiamte-1} and \eqref{eq:pathwise-estiamte-2}, we get
    \begin{align*}
        d(x(t), x(t')) &\leq d(x(t' + \epsilon_{j_0} 2^{-j_0}), x(t')) + \sum_{j = j_0}^{\infty} d(x(t' + \sum_{n=j_0}^{j+1} \epsilon_n 2^{-n}), x(t' + \sum_{n = j_0}^j \epsilon_n 2^{-n})) \\
        &\leq \sum_{j = j_0}^{\infty} 2^{N(\omega)} 2^{-\beta j} \\
        &\leq 2^{N(\omega)} (1 - 2^{-\beta}) (t-s)^{\beta},
    \end{align*}
    and similarly for $d(x(t'), x(s))$. This gives,
    \begin{align*}
        d(x(t), x(s)) \leq 2^{N(\omega) + 1} (1 - 2^{-\beta}) \vert t - s \vert^{\beta}. 
    \end{align*}
\end{proof}

%% file: appendix-B.tex
\section{Properties of analytic semigroups and their generators}\label{app:B}

\begin{lemma}\label{lemma:analytic-semigroup}
    Let $-A$ be the generator of an analytic semigroup on a Banach space $X$, denoted $S(\cdot)$, such that for some $\lambda \geq 0$ the real part of the spectrum of $\lambda + A$ is strictly positive. Then,
    \begin{enumerate}
        \item[(a)] for any $\alpha \geq 0$ there is $C_{\alpha}, \delta > 0$ such that for any $t > 0$ and $u \in X$
        \begin{align*}
            S(t) u \in D((\lambda + A)^{\alpha}) \quad \text{with} \quad \Vert (\lambda + A)^{\alpha} S(t) u \Vert_X \leq C_{\alpha} e^{(\lambda-\delta)t} t^{-\alpha} \Vert u \Vert_X,
        \end{align*}
        with convention $(\lambda + A)^0 = I$, \medskip

        \item[(b)] for $\alpha, \beta \in \mathbb{R}$, $u \in D((\lambda + A)^{\gamma})$ with $\gamma = \max(\alpha, \beta, \alpha + \beta)$,
        \begin{align*}
            (\lambda + A)^{\alpha} (\lambda + A)^{\beta} u = (\lambda + A)^{\beta} (\lambda + A)^{\alpha} u = (\lambda + A)^{\alpha + \beta} u,
        \end{align*}
        
        \item[(c)] for $\alpha \in \mathbb{R}$, $D((\lambda + A)^{\alpha})$ is dense in $X$, \medskip
        
        \item[(d)] for $\alpha \in \mathbb{R}$, $(\lambda + A)^{\alpha} S(t) u = S(t) (\lambda + A)^{\alpha} u$ for $u \in D((\lambda + A)^{\alpha})$, \medskip
        
        \item[(e)] and finally for $\alpha \in [0,1]$,
        \begin{align*}
            \Vert (S(t) - I) u \Vert_X &\leq C e^{\lambda t} t^{\alpha} \Vert (\lambda + A)^{\alpha} u \Vert_X.
        \end{align*}
    \end{enumerate}
\end{lemma}
\begin{proof}
    The first four statements can be found in, e.g., Chapter 2 (Theorem 6.8 and 6.13) in \cite{pazy}. The last one follows by similar arguments as in the case of $\lambda = 0$ (see, e.g., \cite{2024-auestad}). 
\end{proof}

The following lemma describes convergence of the fully discrete approximation $S_{h,\Delta t}(\cdot)$ of $S_1(\cdot)$.
\begin{lemma}\label{lemma:fully-discrete-semigroup-approximation}
    Let $S_{h,\Delta t}(t)$ be given as in \eqref{eq:fully-discrete-semigroup-operator}. Under Assumption \ref{assumption:model} and \ref{assumption:fem}, there are $C, c, \delta > 0$ such that for any $\theta \in [0,2]$, and $\rho \in [-1, \theta] \cap [-2 + \theta, \theta]$
    \begin{align*}
        \Vert (S_1(t) - S_{h,\Delta t}(t) \pi_h ) u \Vert_H &\leq C e^{c (\lambda-\delta) t} t^{-\theta/2 + \rho/2} (h^{\theta} + \Delta t^{\theta / 2}) \Vert (\lambda + A_1)^{\rho/2} u \Vert_H.
    \end{align*}
\end{lemma}
\begin{proof}
    This follows by Theorem 2.14 and 2.24 in \cite{2024-auestad}. For the details, see, e.g., Lemma 4.1 in \cite{2025-auestad}.
\end{proof}

\begin{lemma}\label{lemma:A_h-quadrature}
Let $A_{2,h}$ be as in \ref{cond:discrete-operator}, $Q_k^{-\gamma}$ be given by \eqref{eq:quadrature}, and $\gamma \in (0,1)$. Then, for some $C > 0$ independent of $k, h \in (0,1)$, the following estimate holds,
\begin{align}
    \Vert (A_{2,h}^{-\gamma} - Q_k^{-\gamma}(A_{2,h})) u \Vert_H \leq C e^{-\frac{\pi^2}{2k}} \Vert u \Vert_H, \quad \text{for any } u \in V_h. 
\end{align}
\end{lemma}
\begin{proof}
    This follows by Theorem 3.2 in \cite{2019-bonito}. As stated in Theorem 3.2 in \cite{2019-bonito}, the constant $C$ depends only on $\gamma$ and the continuity and coercivity constant of $(A_{2,h} \cdot, \cdot)_H$ on $(V_h, \Vert \cdot \Vert_{H^1(\mathcal{D})})$, which are the same as that of $a_2$.
\end{proof}